\documentclass[draft]{amsart}

\usepackage{amsthm}
\usepackage[leqno]{amsmath}
\usepackage{latexsym,amsfonts,amssymb}
\usepackage{calligra}
\usepackage[T1]{fontenc}

\usepackage[all,cmtip]{xy} \SelectTips{eu}{} \SilentMatrices
\xyoption{all}
\xyoption{arc}

\usepackage[urlcolor=blue]{hyperref}

\usepackage{color}

\newcommand{\numberseries}{\mdseries}   

\newlength{\thmtopspace}                
\newlength{\thmbotspace}                
\newlength{\thmheadspace}               
\newlength{\thmindent}                  

\renewcommand{\subparagraph}{\vspace{\thmbotspace}}

\newtheoremstyle{bfupright head,slanted body}
{\thmtopspace}{\thmbotspace}
{\slshape}{\thmindent}{\bfseries}{.}{\thmheadspace}
{{\numberseries \thmnumber{\bf #2 }}\thmnote{#3}}

\newtheoremstyle{bfupright head,upright body}
{\thmtopspace}{\thmbotspace}
{\upshape}{\thmindent}{\bfseries}{.}{\thmheadspace}
{{\numberseries \thmnumber{\bf #2 }}\thmnote{#3}}

\newtheoremstyle{bfit head,upright body}
{\thmtopspace}{\thmbotspace}
{\upshape}{\thmindent}{\upshape}{.}{\thmheadspace}
{{\numberseries\thmnumber{\bf #2 }}
  {\bfseries\itshape\thmnote{\negthickspace#3}}}

\newtheoremstyle{it head,upright body}
{\thmtopspace}{\thmbotspace}
{\upshape}{\thmindent}{\upshape}{.}{\thmheadspace}
{{\numberseries\thmnumber{\bf #2 }}
  {\itshape\thmnote{\negthickspace#3}}}

\newtheoremstyle{fixed bf head,slanted body}
{\thmtopspace}{\thmbotspace}{\slshape}
{\thmindent}{\bfseries}{.}{\thmheadspace}
{{\numberseries \thmnumber{\bf  #2 }}\thmname{#1}\thmnote{ (#3)}}

\newtheoremstyle{fixed bf head,upright body}
{\thmtopspace}{\thmbotspace}{\upshape}
{\thmindent}{\bfseries}{.}{\thmheadspace}
{{\numberseries \thmnumber{\bf #2 }}\thmname{#1}\thmnote{ (#3)}}

\newtheoremstyle{fixed bfit head,upright body}
{\thmtopspace}{\thmbotspace}{\upshape}
{\thmindent}{\bfseries\itshape}{.}{\thmheadspace}
{{\numberseries \thmnumber{\bf#2 }}\thmname{#1}\thmnote{ (#3)}}

\newtheoremstyle{sc head,small body}
{\thmtopspace}{\thmbotspace}
{\small\upshape}{\thmindent}{\scshape}{.}{\thmheadspace}
{\thmname{#1}}

\newtheoremstyle{numbered paragraph}
{\thmtopspace}{\thmbotspace}{\upshape}
{\thmindent}{\upshape}{}{0pt}
{{\numberseries \thmnumber{\bf #2 }}}

\newtheoremstyle{unnumbered paragraph}
{\thmtopspace}{\thmbotspace}{\upshape}
{\parindent}{\upshape}{}{0pt}

\theoremstyle{bfupright head,slanted body}
\newtheorem{res}{}[section]             \newtheorem*{res*}{}

\theoremstyle{bfit head,upright body}
                 \newtheorem*{com*}{}

\theoremstyle{bfupright head,upright body}
\newtheorem{bfhpg}[res]{}               \newtheorem*{bfhpg*}{}

\theoremstyle{it head,upright body}
               \newtheorem*{ithpg*}{}

\theoremstyle{sc head,small body}

\theoremstyle{fixed bf head,slanted body}
\newtheorem{thm}[res]{Theorem}          \newtheorem*{thm*}{Theorem}
\newtheorem{prp}[res]{Proposition}      \newtheorem*{prp*}{Proposition}
\newtheorem{cor}[res]{Corollary}        \newtheorem*{cor*}{Corollary}
\newtheorem{lem}[res]{Lemma}            \newtheorem*{lem*}{Lemma}

\theoremstyle{fixed bf head,upright body}
\newtheorem{dfn}[res]{Definition}       \newtheorem*{dfn*}{Definition}
     \newtheorem*{con*}{Construction}
      \newtheorem*{obs*}{Observation}
\newtheorem{rmk}[res]{Remark}           \newtheorem*{rmk*}{Remark}
\newtheorem{exa}[res]{Example}          \newtheorem*{exa*}{Example}
         \newtheorem*{exe*}{Exercise}
            \newtheorem{stp*}{Setup}
            \newtheorem{blk*}{\!\!\!}

\newtheorem{ntn}[res]{Notation}

\theoremstyle{numbered paragraph}
\newtheorem{ipg}[res]{}

\theoremstyle{unnumbered paragraph}
\newtheorem{ipg*}{}

\newlength{\thmlistleft}        
\newlength{\thmlistright}       
\newlength{\thmlistpartopsep}   
\newlength{\thmlisttopsep}      
\newlength{\thmlistparsep}      
\newlength{\thmlistitemsep}     

\newcounter{eqc}
  {\end{list}}%

\newcounter{prt}
  {\end{list}}%

\newcounter{rqm}
  {\end{list}}%

  {\end{list}}%

  {\end{list}}%

\newcounter{exercise}
  {\end{list}}%

\newenvironment{prf*}[1][Proof]{%
  \begin{proof}[\it #1]
    \setcounter{equation}{0}
    \renewcommand{\theequation}{\arabic{equation}}}
  {\end{proof}
}

\newcommand{\pgref}[1]{(\ref{#1})}

\renewcommand{\eqref}[1]{\pgref{eq:#1}}

\numberwithin{equation}{res}
\newcounter{marcom}

\renewcommand{\Im}[1]{\nobreak{\operatorname{Im}#1}}
\newcommand{\coker}{\nobreak{\operatorname{coker}}}
\newcommand{\Ker}[1]{\nobreak{\operatorname{Ker}#1}}
\newcommand{\Ann}{\nobreak{\operatorname{Ann}}}

\renewcommand{\le}{\leqslant}
\renewcommand{\ge}{\geqslant}
\newcommand{\onto}{\twoheadrightarrow}
\newcommand{\lra}{\longrightarrow}
\newcommand{\xra}[2][]{\xrightarrow[#1]{\;#2\;}}
\newcommand{\lla}{\longleftarrow}
\newcommand{\xla}[2][]{\xleftarrow[#1]{\;#2\;}}

\newcommand{\fm}{\mathfrak{m}}

\newcommand{\fc}{\mathfrak{c}}

\newcommand{\Hom}{\operatorname{Hom}}

\newcommand{\ceil}[1]{\lceil#1\rceil}
\newcommand{\floor}[1]{\lfloor#1\rfloor}

\newcommand{\ann}{\operatorname{ann}}

\newcommand{\im}{\operatorname{im}}

\renewcommand{\Gamma}{\textrm{C}}

\newcommand{\h}{\operatorname{H}}

\newcommand{\rank}{\operatorname{rank}\,}
\newcommand{\soc}{\operatorname{soc}}
\newcommand{\syz}{\operatorname{syz}\,}
\newcommand{\Pf}{\operatorname{Pf}}
\newcommand{\iden}{\operatorname{I}}
\newcommand{\Sym}{\operatorname{Sym}}
\newcommand{\reg}{\operatorname{reg}\,}

\newcommand{\linkcomp}[1]{link-#1-compressed}
\newcommand{\qlinkcomp}{\linkcomp{$q$}}

\def\urltilda{\kern -.15em\lower .7ex\hbox{\~{}}\kern .04em}
\def\widebardisplay#1{%
  \setbox0=\hbox{$\displaystyle #1$}
  \dimen0=\wd0%
  \advance\dimen0 by -3.2pt
  \vbox{%
    \nointerlineskip%
    \moveright 1.2pt 
    \vbox{\hrule width \dimen0}%
    \nointerlineskip%
    \kern 1.25pt
    \box0%
  }%
}

\def\widebartext#1{%
  \setbox0=\hbox{$#1$}
  \dimen0=\wd0%
  \advance\dimen0 by -3.2pt
  \vbox{%
    \nointerlineskip%
    \moveright 1.2pt 
    \vbox{\hrule width \dimen0}%
    \nointerlineskip%
    \kern 1.25pt
    \box0%
  }%
}

\def\widebarscript#1{%
  \setbox0=\hbox{$\scriptstyle #1$}
  \dimen0=\wd0%
  \advance\dimen0 by -2pt
  \vbox{%
    \nointerlineskip%
    \moveright 1pt 
    \vbox{\hrule width \dimen0}%
    \nointerlineskip%
    \kern .8pt
    \box0%
  }%
}

\def\widebarscriptscript#1{%
  \setbox0=\hbox{$\scriptscriptstyle #1$}
  \dimen0=\wd0%
  \advance\dimen0 by -1pt
  \vbox{%
    \nointerlineskip%
    \moveright .5pt 
    \vbox{\hrule width \dimen0}%
    \nointerlineskip%
    \kern .6pt
    \box0%
  }%
}

\begin{document}
  
  \allowdisplaybreaks[2]
  
  \title[Betti numbers]{Betti numbers of the Frobenius powers of the maximal ideal over general hypersurfaces}
  
  \author[C.\, Miller  ]{Claudia Miller}
  \address{C.M.\newline\hspace*{1em} Mathematics Department, Syracuse University, Syracuse, NY 13244, U.S.A.}
  \email{clamille@syr.edu} 
  \urladdr{http://clamille.mysite.syr.edu}

  \author[H. Rahmati]{Hamidreza Rahmati}  
  \address{H.R.\newline\hspace*{1em} Department of Mathematics,  University of Nebraska,
    Lincoln, NE 68588,  USA}  
  \email{hrahmati2@unl.edu}

    \author[R.\, R.G.  ]{Rebecca R.G.}
  \address{R.R.G.\newline\hspace*{1em} Department of Mathematical Sciences, George Mason University, Fairfax, VA, U.S.A.}
  \curraddr{George Mason University}
  \email{rrebhuhn@gmu.edu}

  
  
  
  
\thanks{C.\ Miller was partially supported by the National Science Foundation (DMS-1802207), Syracuse University Small Grant, and Douglas R.\ Anderson Faculty Scholar Fund.}

  \begin{abstract}
    The main goal of this paper is to prove, in positive characteristic $p$, stability behavior for the graded Betti numbers in the periodic tails of the minimal resolutions of Frobenius powers of the homogeneous maximal ideals for very general choices of hypersurface in three variables whose degree has the opposite parity to that of $p$. 
    We also find some of the structure of the matrix factorization giving the resolution. 
    We achieve this by developing a method 
    for obtaining the degrees of the generators  of the defining ideal of an 
    $\fc$-compressed Gorenstein Artinian graded algebra from its socle degree, where $\fc$ is a Frobenius power of the homogeneous maximal ideal. 
    As an application, we also obtain the Hilbert-Kunz function of the hypersurface ring, as well as the Castelnuovo-Mumford regularity of the quotients by Frobenius powers of the homogeneous maximal ideal. 
  \end{abstract}

  \maketitle

  \section*{Introduction}
  \label{sec:intro}

Let $R=k[x_1,\dots,x_n]/(f)$ be a standard graded hypersurface ring over a field $k$ with homogeneous maximal ideal $\fm$ and $f\in\fm^2$. 
In this paper we study the question in positive characteristic $p$ of how the free resolutions of the Frobenius powers
\vspace{-2mm}
\[
\fm^{[p^e]}=(x_1^{p^e},\dots,x_n^{p^e})
\] 
are related to each other, in view of the fundamental role that Frobenius powers play in providing invariants of the ring. 
In three variables, Kustin and Ulrich made a fascinating 
computational observation
that the periodic tails of these resolutions seem to eventually cycle as $e$ increases (and in fact sometimes stabilize). 
In the case that $R$ is a diagonal hypersurface, that is, for $f=x_1^{d}+x_2^{d}+x_3^{d}$, this behavior has successfully been determined by 
Kustin, Rahmati, and Vraciu  in \cite{KRV-12}. In addition, they identify for which $p$ and $e$ the tails are infinite (and they prove similar results for the two variable case). See also recent work of Kustin, R.G., and Vraciu for the four variable diagonal hypersurface case \cite{KRGV-22,KRGV-23}.

In this paper we prove that, for odd primes $p$, stability of the graded Betti numbers holds for {\it very general} hypersurfaces $f$ in 3 variables of even degree $d<p$ and for all $q=p^e$. 
In addition, we show that the  corresponding matrix factorizations can be chosen to have very nice structure, namely so that one of the matrices consists of {\it linear} forms and has Pfaffian equal to $vf$ for some unit $v$. 

More precisely, we prove  
these results hold for $f$ satisfying a condition we call \textit{\qlinkcomp}, 
which turns out to be a  \textit{Zariski-open} condition on the coefficients of $f$ for any particular values of $e$ and $d=\deg f$; see Definitions~\ref{def-qlinkcomp} and \ref{relcompressed} and Remark~\ref{rmk-general}.
That the statement holds for very general $f$, that is, that the countable intersection of these Zariski open sets over all $e$ is {\it nonempty} for even $d$ and odd $p$, follows from the doctoral thesis of Camphire \cite{Camphire-thesis}.  

Our result is the following, whose more precise statement is Theorem~\ref{gradedBetti}(a); see also Remark~\ref{rk:q=d+1} for the $q=d+1$ case. 

\begin{bfhpg*}[{Theorem A}]
{\it 
Let $k$ be a field of characteristic $p>0$. Suppose that $p$ is odd and $d$ is even with $p>d$. 
Let $R=k[x,y,z]/(f)$ be a standard graded hypersurface ring for a very general choice of $f$ of degree $d$. 

Then for all $q=p^e$,
the (eventually 2-periodic) minimal graded free resolution of $R/\fm^{[q]}$ is of the form}
\begin{equation*}
    \cdots \xra{\partial_4} R^{2d} 
    \xra{\partial_3} R^{2d}
    \xra{\partial_4} R^{2d}  
    \xra{\partial_3} R^{2d} 
    \xra{\partial_2} R^3 
    \xra{\partial_1} R \to 0
\end{equation*}
{\it 
whose differentials are maps of pure graded degrees 
\[
\deg(\partial_1)=q
, \ 
\deg(\partial_2)=\frac{1}{2}(q+d-1)
, \ 
\deg(\partial_3)=1
, \ 
\deg(\partial_4)=d-1.
\]
Furthermore, there is a matrix factorization $(\widetilde{\partial_3},\widetilde{\partial_4})$ determining the resolution above for which $\widetilde{\partial_3}$ is a linear skew-symmetric matrix over the ring $k[x,y,z]$ with  
\begin{equation*}
    \operatorname{Pf}(\partial_3) = vf
\end{equation*}
for some unit $v\in k$. 

In particular, for any two values $q_1 > q_0$ of $q$, the high graded Betti numbers of the $R$-modules $R/\fm^{[q_0]}$ and $R/\fm^{[q_1]}$ are the same, up to a constant shift of $\frac{3}{2}(q_1-q_0)$. }
\end{bfhpg*}

We prove Theorem A by first finding the relevant minimal graded resolutions over the polynomial ring $P=k[x_1, \cdots, x_n]$; see Proposition~\ref{I-Presn} and Theorem~\ref{gradedBetti}(c).  This also allows us to compute the Hilbert-Kunz function as follows; see Theorem~\ref{HKfunction}. 
The leading term agrees with the results of Buchweitz and Chen in \cite{BuCh-97}. 

\begin{bfhpg*}[{Theorem B}]
{\it 
Suppose that $p$ is odd and $d$ is even with $p>d$. 
For very general choices of $f$ of degree $d$ and any $q=p^e$, the 
Hilbert-Kunz function of $R$ is given by
\begin{equation*}
HK(q)=\frac{3}{4}dq^2-\frac{1}{12}(d^3-d).   
\end{equation*}}
\end{bfhpg*}

In addition, it allows us to compute the Castelnuovo-Mumford regularity of the quotients by Frobenius powers of the homogeneous maximal ideal, giving an affirmative answer, for those hypersurfaces, to the question asked by Katzman \cite{Kat-98} of whether its growth is linear in the power; see Theorem~\ref{gradedBetti}(b). 

\begin{bfhpg*}[{Theorem C}]
{\it Suppose that $p$ is odd and $d$ is even with $p>d$. 
For very general choices of $f$ of degree $d$ and any $q=p^e$, the Castelnuovo-Mumford regularity is given by
\begin{equation*}
\reg(R/\fm^{[q]})=\frac{3}{2}q+\frac{1}{2}d-\frac{5}{2}   
\end{equation*}}
\end{bfhpg*}

It would be very interesting to find what general principle lies behind this stability phenomenon for the Betti numbers and whether this behavior has connections with other invariants of singularities such as $F$-signature and $F$-pure threshold, for which much is still unknown even for hypersurfaces.

The question of stability of the Betti numbers is closely related to another question of independent interest, and indeed answering this other question provides the key ingredient to our proof of Theorem A. 
The question is: What is the behavior, as $e$ increases, of the socles 
\[
\soc\!\left( 
R/(x_1^{p^e},\dots,x_n^{p^e})  
\right)
\]
where the socle of a module is the largest submodule that is a $k$-vector space?  These socles have played a key role in other results, such as in \cite{KV-07}. 
The connection between the two problems is further evident in three variables.
Indeed, Kustin and Ulrich prove that in this case, under some mild hypotheses, the graded Betti numbers of these quotients are the same up to a constant shift if and only if their socles are isomorphic as graded vector spaces up to the same shift; see \cite[1.1]{KU-09}. 

Our proof of Theorem A, which does not use Kustin and Ulrich's result, relies on the following technical result, in which we determine the socle behavior for any particular $q$ and general $f$, 
in fact in any number of variables. The result, as stated below, follows from Theorem~\ref{soc-thm}, and yields the isomorphism of socles in Corollary~\ref{KUsoc}.
Note that in the statement 
\begin{itemize}
    \item $s$ is even when $p$ is odd and $d$ is even (or $p=2$ and $n-d$ is even),  
    \item $s$ is odd when $p$ is odd and $d$ is odd (or $p=2$ and $n-d$ is odd).
\end{itemize}

\begin{bfhpg*}[Theorem D]
{\it Let $k$ be a field of characteristic $p>0$, and fix $q=p^e$. 
Consider the standard graded hypersurface ring 
\[
R=k[x_1,\dots,x_n]/(f)
\]
where $n \geq 3$. Let $d=\deg f$ and assume that $q \geq (n+d)/(n-2)$.

If $f$ is  \textit{\qlinkcomp}, 
the following hold for the socle module
$\soc(R/\fm^{[q]})$ where 
\[
s=n(q-1)-d.
\]
\begin{itemize}
\item[(1)] 
Its generators lie in degrees
\[
s_i=\frac{1}{2}(n(q-1)+d-i)
\qquad 
i \in
\begin{cases}
\{2\}
& {\textrm{ if $s$ is even }} \\
\{1, 3\} 
& {\textrm{ if $s$ is odd}} 
\end{cases}
\]

\item[(2)] 
If $n=3$, its dimension satisfies: 
\begin{align*}
&\qquad {\textrm{ if $s$ is even, }} \dim_k \soc \left(R/\fm^{[q]}\right)_{\!s_2}  \!\!\!= 2d.
\\
&\qquad {\textrm{ if $s$ is odd, }}
\dim_k \soc \left(R/\fm^{[q]}\right)_{\!s_1} \!\!\!= d
{\textrm{ \ and \ }}
\dim_k \soc \left(R/\fm^{[q]}\right)_{\!s_3}\!\!\!\leq 3d.
\end{align*}

\item[(3)] 
If $n \geq 4$, its dimension is equal to a polynomial in $q$ with leading term
\[
\frac{2}{(n-2)!} \left[ \left(\frac{n}{2}\right)^{n-2}\!\!\!\!-n\left(\frac{n-2}{2}\right)^{n-2} \right] q^{n-2} 
\]
for $s$ even and bounded by a polynomial in $q$ of degree $n-2$ for $s$ odd. 
Furthermore, the coefficients of these polynomials depend only on $n$ and $d$. 
\end{itemize}
}
\end{bfhpg*}

Note that in the three variable case the result is much tamer and also more explicit in that the exact dimension is determined and constant when $s$ is even and bounded when $s$ is odd. For $n\geq 4$ variables, the socle dimensions are not nearly as well-behaved, which indicates that the story of high syzygies of Frobenius powers of the maximal ideal is more complicated. 
Indeed, as given in part (3) above, for $s$ even, the socle dimensions grow as a polynomial in $q$.
This result in any number of variables could be helpful in future investigations of Frobenius powers. 
Although it may appear strange at first that for $n\geq 4$ the leading term does not involve $d$ (unlike for $n=3$), note that this leading term vanishes at $n=3$; hence the leading term for $n=3$ is actually the next lower degree term, which does happen to involve $d$. 

We prove Theorem~C by considering the linked ideal $J=((x_1^q, \dots, x_n^q) \colon f)$ in the polynomial ring $P=k[x_1,\dots, x_n]$ as in \cite{KRV-12}, which defines a Gorenstein Artinian algebra. 
By the theory of Macaulay inverse systems, such an algebra is determined by its associated inverse polynomial, which turns out to be very simple in this setting; see Lemma~\ref{phi}. 
To investigate this algebra, we develop a theory to describe some of the structure of $P$-free resolutions of certain truncations of $J$. 

More generally, given a standard graded Gorenstein Artinian algebra $A$, written as a quotient $A=P/J$ of the polynomial ring $P=\Sym_k (A_1)$, we use its associated inverse polynomial to determine 
the format of resolutions of quotients $P/J_{\geq m}$ by truncations of the ideal. 
We do not find the minimal free resolutions explicitly, but identify parts of the resolution as cohomologies of certain subcomplexes of Koszul complexes. 

However, when $A$ is what we call  $\frak{c}$-compressed for a complete intersection ideal $\frak{c}$ (see Definition~\ref{relcompressed}), 
we prove that the first few such cohomologies vanish for $m\ge s/2$. 
When this is the case in our setting for $P/J$ with $J=((x_1^q, \dots, x_n^q) \colon f)$ we say that $f$ is \qlinkcomp\ (see Definition \ref{def-qlinkcomp}). 
In any case, this yields enough information that we can read off the degrees of the generators of $J$, and so the application we have in mind goes through. 

Our notion of $\frak{c}$-compressed  is a special case of relatively compressed algebras. Relatively compressed algebras are a useful variant of compressed algebras, which have been studied for many years; see Definition~\ref{oldrelcompressed} for the precise definition, which originally appeared in \cite{IarKan-99, MMN-05}. 
However, not much was known about their resolutions, or even generators.  We extend the methods for {\it compressed} algebras of El Khoury and Kustin in \cite{ElkhKus-14} and Miller and Rahmati in \cite{MiRa-18}
to find resolutions, neither minimal nor completely explicit, of $P/J_{\ge m}$ for every $m$, in  any characteristic; see Theorem~\ref{resn-general}. When $P/J$ is {\it $\frak{c}$-compressed} and $m\ge s/2$ where $s$ is the socle degree of $P/J$, the beginning of the resolution becomes more explicit. 
Section~\ref{sec:relcomp} is dedicated to these results.

In Section~\ref{sec:socle} we use the methods from the previous section to find the number and degrees of generators of the link $J=((x_1^q, \dots, x_n^q) \colon f)$ whenever $P/J$  is a $(x_1^q,\ldots,x_n^q)$-compressed Gorenstein algebra and then apply linkage theory to determine the desired socle degrees for Theorem D; 
see Remark~\ref{muofJ}. 
In Section 5, we prove our main results, Theorems A, B, and C, using the technical input from the previous section. 

We also note the curious fact that the behavior of the resolutions of the Frobenius powers for a \qlinkcomp\ hypersurface can differ from that in the diagonal hypersurface case found in \cite{KRV-12}. For some degrees, the diagonal hypersurface case is  $(x_1^q,\ldots,x_n^q)$-compressed, and for some it is not; see Example~\ref{diagonalexample}. 
So, while we were inspired by the results of Kustin, Rahmati, and Vraciu, our results are a complementary case, rather than a generalization of their results.

The paper is organized as follows.  In Section~\ref{sec:bkgrnd} we review the theory of inverse systems and establish some useful properties and in Section~\ref{sec:compressed} we prove some preliminary results about algebras that are $\frak{c}$-compressed for a complete intersection ideal $\frak{c}$.  
Resolutions of such algebras can be found in Section~\ref{sec:relcomp}. 
Then in Section~\ref{sec:socle} we apply this general theory of resolutions  to find the number and degrees of generators of the link $J=((x_1^q, \dots, x_n^q) \colon f)$ and from this we derive our result for the socles of $R/\fm^{[q]}$. In Section~\ref{sec:frobenius}, for the $(x^q,y^q,z^q)$-compressed case in 3 variables, we compute some of the structure of the resolutions of the Frobenius powers over the hypersurface ring and the  Hilbert-Kunz function of the hypersurface ring,  as well as the regularity of the quotient by Frobenius powers of the maximal ideal. Section~\ref{sec:examples} includes examples illustrating the stability behavior in both the \qlinkcomp\ hypersurface case and the diagonal hypersurface case.

All the algebras occurring in this paper are assumed to be standard graded, and we set $P=k[x_1, \cdots, x_n]$ throughout. We will use $f$ to denote a nonzero homogeneous element of $P$.

\section{Inverse systems}
\label{sec:bkgrnd}
  

In 1918 Macaulay described a one-to-one correspondence between 
Gorenstein Artinian algebras of the form $P/I$, where $P=k[x_1, \cdots, x_n]$, and cyclic 
$P$-submodules of the ring of inverse polynomials $k[x_1^{-1},\dots,x_n^{-1}]$. 
Later Emsalem \cite{Ems-78} and Iarrobino \cite{Iar-84} extended this correspondence to all Artinian algebras $P/I$ and finitely generated $P$-submodules of the ring of inverse polynomials. As such the correspondence can be viewed as a version of Matlis duality in view of \cite{Nor-74}. 
In characteristic zero, they can equivalently be described using the ring of partial derivative operators. 
All of these are now called Macaulay's inverse systems. See also \cite[2.6]{Eise-95} and \cite{IarKan-99}. 

We describe this correspondence in the graded setting, in the Gorenstein case. We use the divided power ring in place of the inverse polynomial ring since they are isomorphic as $P$-modules. 
Consider a homogeneous Gorenstein ideal $I$ primary to the 
homogeneous maximal ideal $(x_1, \dots, x_n)$
so that $A=P/I$ is a graded Artinian Gorenstein algebra. 

Every such ideal $I$ has an 
explicit description as an annihilator of ``divided power polynomials'' as follows. 
Let $V$ be a vector space of dimension $d$ with basis $x_1, \dots, x_n$. 
Let $S(V)$ denote the symmetric $k$-algebra which can be identified with 
$P=k[x_1, \cdots, x_n]$ and let $D(V^*)$ denote the divided power $k$-algebra; that is,
\begin{equation*}
D(V^*)
=\Hom_{{\textrm{gdd-}}k}(S(V),k) 
=\oplus_{i \ge 0} \Hom_{k}(S(V)_i,k)
\end{equation*}
which is a $k$-vector space on the dual basis of monomials 
$x_1^{(i_1)} \cdots x_n^{(i_n)}$ where $x_i^{(0)}$ denotes 1 for all $i$. 
Recall that $D(V^*)$ is a module over $S(V)$ via  
\[
x_1^{j_1} \cdots x_n^{j_n} \cdot (x_1^{(i_1)} \cdots x_n^{(i_n)})
=\begin{cases}
x_1^{(i_1-j_1)} \cdots x_n^{(i_n-j_n)} & {\textrm{if } ~ i_n \ge j_n ~\textrm{for all }} n \\
0 & {\textrm{else}} 
\end{cases}
\]

\begin{rmk}\label{inverse}
In place of $D(V^*)$, one could just as well use the ring of inverse polynomials 
\[
k[x_1^{-1}, \dots, x_n^{-1}]
\]
with the analogous $S(V)$-module structure, that is, with 
$x_i^{\alpha} x_i^{-\beta} = 0$ whenever $\alpha > \beta$. 
Note that  $k[x_1^{-1}, \dots, x_n^{-1}]$ can be identified with the injective hull $E(k)$ of $k$ over $P$, as proved by Northcott in \cite{Nor-74} and so the ideal obtained from an inverse system is actually the annihilator of a finitely generated $k$-submodule of the injective hull. 
\end{rmk}

We now describe the one-to-one correspondence 
between homogeneous Gorenstein ideals $I$ of $P$ with  Artinian and cyclic  
submodules $\langle \varphi \rangle$ of the $P/I$-module  $D(V^*)$.

Given a homogeneous element $\varphi$ of  $D(V^*)$, let 
\[
I(\varphi) = (0\colon_{S(V)} \varphi )
= \{r\in S(V) ~|~ r \varphi=0 \},
\]
which is a graded Artinian ideal of $P$; 
indeed, all monomials of degree greater than 
$s=\deg \varphi$ are in $I$ 
and so $(P/I)_{\ge s}=0$.  

Conversely, given a homogeneous Artinian Gorenstein ideal $I$  of $P$, let 
\[
I^\perp
= (0:_{D(V^*)}I) = \{ \varphi \in D(V^*) \mid I \varphi = 0  \}
\]
It turns out that since $I$ is Artinian Gorenstein, 
$I^\perp$ is a cyclic $S(V)$-submodule $\langle \varphi \rangle$ of $D(V^*)$. 
Such an element $\varphi$ is called an inverse system for $I$. 

This correspondence between Artinian Gorenstein ideals in $P$ and cyclic submodules of $D(V^*)$ was given by Macaulay. It was extended to the non-cyclic finitely-generated submodules in \cite{Ems-78} and  \cite{Iar-84}. 

\begin{exa}
\label{ex-classic}
Taking for the inverse system the divided polynomial $\varphi=x^{(2)}+y^{(2)}+z^{(2)}$
one obtains the ideal 
$I=\ann \varphi = (xy,yz,xz,x^2-y^2,x^2-z^2)$.
\end{exa}

The main result in the theory of inverse systems is that 
these assignments are inverse to each other and that the type of the resulting algebra $P/I$ is equal to the minimal number of generators of the submodule $I^\perp$. In particular, any homogeneous Artinian Gorenstein ideal can be obtained in this way using a single homogeneous $\varphi \in D(V^*)$. 
The degree of $\varphi$ gives the socle degree of the resulting Artinian algebra.

\begin{ipg}
\label{pairings}
We now lift the machinery of the inverse system from $k$-vector spaces to free $P$-modules. 
Let $F$ be a free $P$-module of rank $d$ with basis $X_1, \dots, X_n$. 
We let $S=S(F)$ and $D=D(F^*)$ denote the symmetric $P$-algebra and  the divided power $P$-algebra, respectively. 
We also set $\wedge$ to denote the exterior $P$-algebra $\wedge F$. 
Note that $D$ has an  $S$-module structure analogous to the one described in Remark \ref{inverse}. 
Since $k \to P$ is faithfully flat, properties for the $k$-vector spaces $S(V)$ and $D(V^*)$ transfer to the free $P$-modules $S$ and $D$ in the obvious way. 

Given any homogeneous inverse polynomial in 
$D(V^*)$, we can think of it as a polynomial in $D$ 
by replacing each $x_i$ by $X_i$. 
It should not cause any confusion to call it 
$\varphi$ in either case.  
For any such inverse polynomial   
we get a collection of maps 
  \begin{equation*}
\xymatrixrowsep{2.2pc}\xymatrixcolsep{.7pc}
    \xymatrix{
      &S_0 \ar[d]^{\Phi_0}&S_1 \ar[d]^{\Phi_1}& \cdots  & S_i \ar[d]^{\Phi_i} & \cdots & S_s \ar[d]^{\Phi_s}& S_{s+1} \ar[d] ^{\Phi_{s+1}}&\cdots\\
      &D_{s} &D_{s-1}& \cdots & D_{s-i} & \cdots & D_{0} & 0 &\cdots
  }
  \end{equation*}
defined by
\[
\Phi_i(g)=g\varphi, ~\text{for}~g \in S_i
\]
and where $s=\deg(\varphi)$. For more detail see \cite{MiRa-18}. 
\end{ipg}

\begin{dfn} 
\label{dfn-D'}
Given an inverse polynomial $\varphi$,   we define 
\[
D'_{s-m} = \im \, \left( S_m \xra{\Phi_m} D_{s-m} \right)
\]
Note that each $D_{s-m}'$ is a free direct summand of $D_{s-m}$ as $P$-modules. 
Note further that $\rank D_m'=\rank D_{s-m}'$ by Lemma~\ref{lem-duals} and since dual maps have the same rank.
\end{dfn}

\begin{lem}
\label{lem-duals}
\cite{MiRa-18}
Let $\varphi $ be a divided power polynomial of degree $s$. The induced maps  $\Phi_i \colon S_i \to D_{s-i}$ satisfy the following properties  
  for every integer $i$,
\begin{enumerate}
\item  $\Hom_R(\Phi_i,P) = \Phi_{s-i}$
\item We have $\ker(\Phi_i) \cong \text{coker}(\Phi_{s-i})$ by changing each exponent in every monomial in the kernel to a divided power.
\end{enumerate}  
\end{lem}

\section{$\frak{c}$-compressed algebras}
\label{sec:compressed}

In this section, we recall 
the definitions of compressed and relatively compressed Artinian algebras, define our new notion of $\frak{c}$-compressed algebras for a complete intersection $\mathfrak{c}$, and translate this definition to certain properties of the maps $\Phi_i$ described in the previous section for the corresponding inverse system. Using this, we give several criteria for an ideal in the polynomial ring $P=k[x_1, \cdots, x_n]$ to be $\frak{c}$-compressed, especially when $\mathfrak{c}$ is a Frobenius power of the maximal ideal $(x_1,\ldots,x_n)$.

First we recall the definition of compressed (but only for the case of Gorenstein algebras) so that the reader can compare the two notions. 
Recall that the Hilbert function of a graded algebra $B$ is defined to be $H_i(B)=\dim_k B_i$. 

  \begin{dfn}
    A graded Artinian Gorenstein algebra $P/J$ whose socle is of degree $s$ is said to be  \emph{compressed} if it satisfies
    \begin{equation}
    \label{eqn-compressed}
  H_i(P/J)=\min\{H_i(P),H_{s-i}(P)\}
    \end{equation}
    for $i=1, \dots, s$.  
	\end{dfn}

\begin{ipg}
Iarrobino, \cite[3.4]{Iar-84}, see also Boij and Laksov \cite[3.4]{BoLak-94}, proves  that being compressed holds generally, even in the non-Gorenstein level case. 
Specifically, for general choices of inverse system of divided power polynomials of the same degree, 
the corresponding algebra is compressed. 
Note that the Hilbert function of a compressed algebra is maximal among all the algebras with the same socle degrees.
\end{ipg}

In this paper, the algebras we study are not  compressed, but  satisfy a more general condition which we call $\frak{c}$-compressed, defined below.

\begin{dfn}
\label{relcompressed}
Let $\fc \subseteq P$ be a homogeneous complete intersection ideal such that $P/\fc$ is Artinian. 
Let $J \subseteq P$ be a homogeneous Gorenstein ideal containing $\fc$. 
We say that the algebra $A=P/J$, or equivalently the ideal $J$, is \textit{ $\frak{c}$-compressed} if, for every $i$,  $H_i(P/J)$ takes on the maximum possible value, i.e.,
\[
H_i(P/J)=\min\{H_i(P/\fc),H_{s-i}(P/\fc)\}
\]
where $s$ is the degree of the socle of $A$. 

Since $H_i(P/\fc)$ increases as $i$ increases, the minimum in the formula is $H_i(P/\fc)$ for $i \le \frac{s}{2}$ and $H_{s-i}(P/\fc)$ for $i \ge \frac{s}{2}$.
In particular, note that the condition implies that $J_i=\fc_{i}$ in degrees $i\leq s/2$. 

We say that $A$ is {\it $\fc$-compressed in degree $i$} if the equality above holds for that particular value of $i$. 
\end{dfn}

Note that this notion is stronger than the relatively compressed condition, explored for monomial complete intersection ideals by Iarrobino and Kanev in \cite{IarKan-99} and introduced formally by Migliore, Mir\'o-Roig, and Nagel as a generalization of the compressed condition that has been studied in many places. We recall their definition in the Gorenstein case. 
\begin{dfn}\cite{MMN-05}
\label{oldrelcompressed}
A standard graded Artinian Gorenstein 
algebra $A=P/J$ with socle degree $r$ is said to be \textit{relatively compressed with respect to a complete intersection} $\fc$ if the function $H_i(P/J)$ 
is maximal among Gorenstein ideals $P/J'$ where $\fc \subseteq J'$ and $J'$ has socle degree $r$.
\end{dfn}

This condition does not necessarily imply that the algebras that are relatively compressed with respect to $\fc$  reach the upper bound as given in the definition of $\fc$-compressed; see Remark \ref{rmk-general}. 

In the next lemma, we give an interpretation of the $\fc$-compressed condition
in terms of the maps $\Phi_i$ for the corresponding inverse system.
Further, when the complete intersection is $(x_1^q,\ldots,x_n^q)$ we give an even more explicit version.
Recall from \ref{dfn-D'} that $D'_{s-m} = \im \Phi_m$ is a free summand of $D'_{s-m}$. 

\begin{lem}
\label{TFAErelativelycompressed}
Let $\fc \subseteq P$ be a homogeneous complete intersection ideal such that $P/\fc$ is Artinian and let $C$ be the lift of $\fc$ to the symmetric algebra $S$ obtained by replacing each $x_i$ by $X_i$. 
Let $J \subseteq P$ be a homogeneous Gorenstein ideal containing $\fc$  with inverse polynomial $\varphi$ and let $s$ be the degree of the socle of $P/J$.  The following conditions are equivalent.
\begin{enumerate}
\item $J$ is $\fc$-compressed.
\item 
For every $i \le \frac{s}{2}$ (equivalently for every $i \ge \frac{s}{2}$), $J$ is $\fc$-compressed in degree~$i$.
\item For $i \le \frac{s}{2}$, the kernel of the map $\Phi_i\colon S_i \to D_{s-i}$ equals the $i$th graded piece of the ideal $C$.  
\item 
For $i \ge \frac{s}{2}$, the cokernel of the map $\Phi_i\colon S_i \to D_{s-i}$ is generated minimally by the set $C_{s-i}'$ of  polynomials in $D_{s-i}$ obtained from a minimal generating set of polynomials in $C_{s-i}$ by changing all powers to divided powers.

\item For every $i$ (equivalently for every $i \le \frac{s}{2}$ or for every $i \ge \frac{s}{2}$) one has an equality 
\[
\rank D_i'
=\min\{\rank (S_i/C_i), \rank (S_{s-i}/C_{s-i})\}
\]
\end{enumerate}
In particular, when $\fc=(x_1^q,\ldots,x_n^q)$, and so $C=(X_1^q,\ldots,X_n^q)$, the conditions (3) and (4) translate to 
\begin{enumerate}
\item[($3'$)] 
For $i \le \frac{s}{2}$ the kernel of the map $\Phi_i\colon S_i \to D_{s-i}$ is generated by the monomials $X_1^{a_1} \cdots X_n^{a_n}$ in $S_i$ with $a_j \ge q$ for some  $1 \le j \le n$.
\item[($4'$)] 
For $i \ge \frac{s}{2}$, the cokernel of the  map $\Phi_i\colon S_i \to D_{s-i}$ is generated by the monomials $X_1^{(a_1)} \cdots X_n^{(a_n)}$ in $D_{s-i}$ with $a_j \ge q$ for some $1 \le j \le n$.
\end{enumerate}
\end{lem}

\begin{proof}
First note that  
\begin{align}\label{ranks}
\begin{split}
\dim_k J_i &= \rank \ker \Phi_i
\\
H_i(P/J) &= \rank (S_i/\ker \Phi_i) = \rank D'_{s-i}
\\ 
H_i(P/\fc) &= \rank (S_i/C_i)
\\
\rank D_i' &= \rank D'_{s-i}
\end{split}
\end{align}
where the first two follow from (\ref{pairings}), the third is obvious, and the last one is from (\ref{dfn-D'}) and the fact that dual maps have the same rank. 
Therefore we also get $H_i(P/J)=H_{s-i}(P/J)$. 
Hence (1), (2) and (5) are equivalent since the relevant equations are unchanged if we replace $i$ with $s-i$.
Furthermore, (3) is equivalent to (4) by the duality of the maps $\Phi_i\colon S_i \to D_{s-i}$ and $\Phi_{s-i}\colon S_{s-i} \to D_i$.

For the proof of (5) implies (3), if $i \le \frac{s}{2}$ we have $\rank D_i'
=\rank (S_i/C_i)$ by hypothesis. 
So since $\rank D_i'= \rank (S_i/\ker \Phi_i)$ and $C_i \subseteq \ker \Phi_i$, we get that $C_i = \ker \Phi_i$. 
Lastly note that (3) implies (2) by the second and third formulas in (\ref{ranks}).
\end{proof}

\begin{rmk}
\label{rmk-TFAEsinglei}
The equivalences (2)--(5) in Lemma~\ref{TFAErelativelycompressed} remain valid for a fixed $i\leq s$ as long as one replaces ``for any $i \geq \frac{s}{2}$'' with ``for  $s-i$''. 
\end{rmk} 

For the rest of this section, assume that $\fc=(x_1^q,\ldots,x_n^q)$. 
First we give a concrete version of (5) in Proposition~\ref{prp-D'basis} that is similar to ($3'$) and ($4'$). 
Then we use it to show that, in that case, it is actually only necessary to confirm the $\fc$-compressed condition (or the equivalent condition) in one degree; see Lemma \ref{lem-surjective} and Corollary \ref{relativelyquick} for details.

This requires us to first obtain the inverse polynomial of $J$, which turns out to be surprisingly simple: 

\begin{lem}
\label{phi}
For any Gorenstein ideal $J \subseteq P$ that contains $\fc=(x_1^q,\ldots,x_n^q)$,  
there exists a homogeneous $f\in P$ of degree $d$ with $1<d<q$ such that $J$ is of the form  
\[
J=((x_1^q,\ldots,x_n^q):f)
\]
The inverse polynomial of $J$ is 
\[
\varphi=f[x_1^{(q-1)} \cdots x_n^{(q-1)}]
\]
where multiplication on the left by $f$ is via the $S$-module structure on $D$. 
Note that $\varphi$ can also be written in inverse variables as   
\[
\varphi=\frac f{(x_1^{q-1} \cdots x_n^{q-1})}
\]
as in Remark~\ref{inverse}. 
In particular, its degree is 
\[
\deg \varphi 
= s 
= n(q-1)-d 
{\textrm{ \ \ \ where \ \ }}
d=\deg f
\]

Conversely, any $\varphi$ that is a linear combination of divided power monomials of degree $s<n(q-1)$ with each variable having power strictly less than $q$ can be written in the form above and so provides an inverse system whose associated ideal $J$ is as above. 
\end{lem}

\begin{proof} 
The first part follows from the theory of linkage. 

It suffices to show that $\Ann_P (\varphi)=J=((x_1^q,\ldots,x_n^q):f)$. Suppose that $g \in \Ann_P (\varphi)$. This is equivalent to $gf[x_1^{(q-1)} \cdots x_n^{(q-1)}]=0$ in the divided power ring. This holds if and only if $gf$ is a polynomial where each term has a factor of $x_i^a$ with $a \ge q$ for some $i$. This is true if and only if $gf \in (x_1^q,\ldots,x_n^q)$, if and only if $g \in J_q$.

The last statement follows from the form of $\varphi$ given above. 
\end{proof}

Recall that an ideal $J$ is $\fc$-compressed in degree $i$ if the equality
\begin{equation*}
H_i(P/J)=\min\{H_i(P/\fc),H_{s-i}(P/\fc)\}
\end{equation*}
holds for that integer $i$. 
By the arguments in the proof of Lemma~\ref{TFAErelativelycompressed} (see Remark~\ref{rmk-TFAEsinglei}) this is equivalent to the condition that the rank of the  induced map $\Phi_i:S_i \to D'_{s-i}$  is equal to  $\min\{H_i(P/\fc),H_{s-i}(P/\fc)\}$ for that integer $i$.
For any fixed $i\leq \frac{s}{2}$, or $i \geq \frac{s}{2}$, respectively, it is also equivalent to parts (3) and (4), respectively, of Lemma~\ref{TFAErelativelycompressed} . 

The following consequence of Lemma~\ref{phi} will turn out to be very useful later.

\begin{prp}
\label{prp-D'basis}
Let $\fc=(x_1^q,\ldots,x_n^q)$. 
For any $a \geq s/2$, $J_q$ is $\fc$-compressed in degree $a$
if and only if the free submodule $D'_{s-a} \subseteq D_{s-a}$ has basis given by all the inverse monomials $X_1^{(a_1)}\cdots X_n^{(a_n)}$ of degree $s-a$ such that $a_j < q$ for all $j=1,\dots,n$. 
\end{prp}

\begin{proof}
Let $a\geq s/2$ and we wish to prove that $D'_{s-a}$ is equal to 
\begin{equation*}
Y_{s-a}=k\left\{X_1^{(a_1)} \cdots X_n^{(a_n)} \mid \sum a_i=s-a  \text{ and } a_j<q \text{ for all } 1 \le j \le n\right\}
\end{equation*}
Recall that $D_{s-a}'$ consists of the elements of $D_{s-a}$ that can be written in the form $G\Phi$, where $G \in S_{a}$. 
Recall by Lemma~\ref{phi} that 
$\varphi=f[{x_1^{(q-1)} \cdots x_n^{(q-1)}}]$.
Letting $F$ denote the polynomial obtained from $f$ by replacing $x_i$ with written in $X_i$, 
one has $\Phi=F[{X_1^{(q-1)} \cdots X_n^{(q-1)}}]$. 

The inclusion
$D_{s-a}' \subseteq Y_{s-a}$ is evident as $G\Phi$ has degree at most $(q-1)$ in each $X_i$, since both $G$ and $F$ have only positive powers of the variables $X_i$.
On the other hand, by Lemma \ref{TFAErelativelycompressed} and Remark~\ref{rmk-TFAEsinglei}, the ideal $J_q$ is $\fc$-compressed in degree $a$ if and only if the divided power monomials $X_1^{(a_1)} \cdots X_n^{(a_n)}$ with some $a_i \ge q$ form a basis for the free module $\coker \, \Phi_{a} = D_{s-a}/D_{s-a}'$. Hence $D_{s-a}'$ has the same vector space dimension as $Y_{s-a}$ in this case, and we must have $D_{s-a}'=Y_{s-a}$ whenever $J_q$ $\fc$-compressed in degree $a$.
\end{proof}

For the sake of brevity we introduce the following term, which we will use in our results in Sections \ref{sec:socle}, \ref{sec:frobenius}, and \ref{sec:examples}.

\begin{dfn}
\label{def-qlinkcomp}
Let $f \in P$ be homogeneous, and set $J=((x_1^q,\ldots,x_n^q):f)$.  
We say that $f$ is \textit{\qlinkcomp} if the ideal $J$ is $(x_1^q,\ldots,x_n^q)$-compressed (see Definition \ref{relcompressed} for the latter).
\end{dfn}

\begin{rmk} 
\label{rmk-general}
We claim that, for every polynomial $f$ of fixed degree $d$ and  every fixed $q=p^e$, being \qlinkcomp\ is a Zariski open condition on the coefficients of $f$.  
Indeed, by Lemma~\ref{TFAErelativelycompressed}, this is equivalent to the maps $\Phi_i\colon S/(x_1^q,\ldots,x_n^q) \to D$ induced by the inverse polynomial of $J$ having maximal possible rank for all $i \leq s/2$ where $s=n(q-1)-d$. This is an open condition on the entries of matrices giving these maps (via the maximal minors), but these entries are just the coefficients of $\varphi$. 
But by Lemma~\ref{phi} the inverse polynomial of $J$ is 
\[
\varphi=f[x_1^{(q-1)} \cdots x_n^{(q-1)}]
\]
and so the coefficients of $\varphi$ are exactly the coefficients of $f$, and we are done. 

So, since the condition is an open one, for it to hold generally, it suffices to determine existence, that is, the Zariski open set of coefficients is nonempty. 
It does distinctly fail in certain cases. 
One of these can be derived from an example described by Migliore, Mir\'o-Roig, and Nagel in \cite[Ex 2.12]{MMN-05} in which general choices of a Gorenstein ideal in 3 variables of socle degree 8 containing the 3 quartic generators of a complete intersection ideal $\fc$ are not $\fc$-compressed because the Hilbert function is not the expected one. (These are in fact the link of a polynomial $f$ of degree $3(4-1)-8=1$. However, we can also find many nonlinear examples of such $f$ using \texttt{Macaulay2} \cite{M2}, for example, for $d=2, p=2, e=1,2$.)

For our results on the stability of Betti numbers to hold for all $q$, one thus needs very general choices of $f$ to be \qlinkcomp, that is, choices of $f$ whose coefficients lie in a nonempty countable intersection of Zariski open sets of polynomials of some fixed degree $d$. So, if $k$ is an algebraically closed field, this would mean that ``random choices'' of $f$ should work.

Fortunately, in the setting considered in our paper, Camphire has proved that for each even degree $d$, there is an example of a hypersurface in 3 variables that is \qlinkcomp\ for {\it all} odd $p>d$ and all powers $q$; see \cite{Camphire-thesis}; hence the intersection of Zariski open sets where our results hold is indeed nonempty.
\end{rmk}

\begin{rmk}
\label{relcompforallq}
A stronger question one might ask is: If the Gorenstein ideal $((x_1^q,\ldots,x_n^q) \colon f)$ is $(x_1^q,\ldots,x_n^q)$-compressed for some value of $q$, does the same hold for all larger $q$? 
This is a relatively concrete question in view of 
Lemma~\ref{phi} and Proposition~\ref{prp-D'basis}.

Unfortunately, the answer is negative; in the work of Camphire mentioned in Remark \ref{rmk-general} \cite{Camphire-thesis}, he has found examples of hypersurfaces in 3 variables that are \qlinkcomp\ for some value of $q$ but not for $pq$.
\end{rmk}

As a brief aside, we use Proposition~\ref{prp-D'basis} to obtain the following lemma whose corollary gives a shortcut to deciding whether an ideal is $\fc$-compressed. This generalizes a basic and useful lemma from the compressed case \cite{MiRa-18}.

\begin{lem}
  \label{lem-surjective}
Let $\fc=(x_1^q,\ldots,x_n^q)$ and let $\varphi$ be a divided power polynomial of degree $s$. 
If an ideal $J$ is $\fc$-compressed in degree $m$ for some integer $m\ge  \frac{s}{2}$ then $J$ is $\fc$-compressed in degree $m+i$ for every integer $i \ge 0$.
\end{lem}
\begin{proof}
Suppose that $J$ is $\fc$-compressed in degree $m$ for some $m\ge  \frac{s}{2}$. We show   that the rank of $\Phi_{m+i}:S_{m+i} \to D'_{s-m-i}$  is equal to 
\[
\min\{H_{m+i}(P/\fc),H_{s-m-i}(P/\fc)\}
\]
for all $i \geq 0$. Set
\[
 Y_a = k\left\{X_1^{(a_1)} \cdots X_n^{(a_n)} \mid\sum a_i=a  \text{ and } a_j<q \text{ for all } 1 \le j \le n\right\}.
\] 
and recall that $D_a' \subseteq Y_a$ as explained in the proof of Proposition~\ref{prp-D'basis}.
By the same proposition, this containment is an equality for $a=s-m$. 
We show that  $Y_{s-m-i}$ is in the image of $\Phi_{m+i}$ which implies that $D_{s-m-i}' = Y_{s-m-i}$. 
We do this by forming the  composition of maps of free $P$-modules
\[
\sigma_{m+i} \colon 
Y_{s-m-i} \xra{}Y_{s-m}=D'_{s-m} \xra{\sigma_m} S_m \xra{} S_{m+i}
\] 
where the first map sends any divided power monomial $X_1^{(a_1)}\cdots X_n^{(a_n)}$ to the monomial $X_1^{(a_1+b_1)} \cdots X_n^{(a_n+b_n)}$, where $b_j=\min\{q-a_j-1,i-\sum_{t=1}^{j-1} b_t\}$, the map $\sigma_m$ is a splitting of $\Phi_m$, 
and the last map sends a monomial $X_1^{a_1} \cdots X_n^{a_n}$ to $X_1^{a_1+b_1} \cdots X_n^{a_n+b_n}$, with the $b_j$ chosen as above. Note that, since the exponents remain less than $q$ under the first map, the image lands in $Y_{s-m}=D_{s-m}'$.

We show that $\sigma_{m+i}$ gives a preimage for each monomial in $Y_{s-m-i}$:
\[
\begin{aligned}
\Phi_{m+i}\circ \sigma_{m+i}(X_1^{a_1} \cdots X_n^{a_n}) &=\Phi_{m+i} \left(X_1^{b_1} \cdots X_n^{b_n}(\sigma_m(X_1^{(a_1+b_1)} \cdots X_n^{(a_n+b_n)} )\right)\\
&=\Phi_{m+i}\left(\sigma_m\left(X_1^{(a_1+b_1)} \cdots X_n^{(a_n+b_n)}\right)(X_1^{b_1} \cdots X_n^{b_n})\right) \\
&= \left(\Phi_m \sigma_m(X_1^{(a_1+b_1)} \cdots X_n^{(a_n+b_n)})\right)(X_1^{b_1} \cdots X_n^{b_n}) \\
&= (X_1^{(a_1+b_1)} \cdots X_n^{(a_n+b_n)}) \cdot (X_1^{b_1} \cdots X_n^{b_n}) \\
&=X_1^{(a_1)} \cdots X_n^{(a_n)}
\end{aligned}
\]
where the second equality follows from the commutativity of the symmetric algebra $S$ and the third one holds because $D$ is an $S$-module.
Hence $\Phi_{m+i}$ surjects onto $Y_{s-m-i}$ and so indeed $D_{s-m-i}'=Y_{s-m-i}$. 
So we have in fact constructed a $P$-module homomorphism $\sigma_{m+i}$ splitting $\Phi_{m+i}$ since these monomials form a basis of $D_{s-m-i}'$. 
\end{proof}

Using Lemma~\ref{lem-duals} and Lemma~\ref{TFAErelativelycompressed}, we get the following as an immediate corollary. 

\begin{cor}
  \label{relativelyquick}
Let $\fc=(x_1^q,\ldots,x_n^q)$ and let $\varphi$ be a divided power polynomial of degree $s$. 
If $J_q$ is $\fc$-compressed  in degree $ \lceil \frac{s}{2} \rceil$
then $J_q$ is $\fc$-compressed. 
\end{cor}

  \section{Resolutions of $\fc$-compressed ideals}
  \label{sec:relcomp}

In this section, given an ideal $J$ of $P=k[x_1,\dots, x_n]$ that is $\fc$-compressed for a complete intersection ideal $\fc$, our goal is to determine the form of the resolution of and the (degrees of the) minimal generators of $J$. We describe all possible degrees of the generators of an ideal that is not necessarily $\fc$-compressed; see Corollary~\ref{cor-degrees}. We do this using Macaulay's inverse systems.
We begin by giving an outline of the process that we use.

First we see immediately from the definition of $\fc$-compressed that 
in degree $n \leq \frac{s}{2}$ the only generators are the  generators of $\fc$.

Next, in order to obtain the generators in degrees $n\ge \frac{s}{2}$ we expand the methods in  \cite[Section 6]{MiRa-18}
to obtain a resolution of the truncated ideal $J_{\ge \frac{s}{2}}$, that is, the ideal $J_q$ in degrees greater than or equal to $\frac{s}{2}$. The beginning of the resolution turns out to be explicit enough to read off the degrees, and if $\deg(f)$ is even, the number of the generators in degrees $n\ge \frac{s}{2}$. 

Unlike in Section 6 of that paper, in this setting we are able to prove that the images of certain maps split out of the target. This enables us to simplify the resulting multi-complex at each step by taking kernels and cokernels. In order to explain this modified process, we discuss the construction more concretely. Subsequently, we show that some of these cokernels vanish, and hence we determine the degrees of the generators of $J$, the defining ideal of the algebra.

Before we begin constructing the resolution, we introduce the basic building blocks involved and determine that they are free.
\begin{dfn}
\label{dfn-LK}
In \cite{BuchEis-75} Buchsbaum and Eisenbud construct minimal free resolutions of powers of  ideals that are generated by regular sequences. 
We recall the special case of their construction for powers of the homogeneous maximal ideal $\mathfrak m$ of $P$. Here all tensor products are over $P$.

We begin by describing the free $P$-modules involved in their resolutions. They arise as kernels in strands of the general Koszul complex $K(X_1, \dots, X_n;S)$ on the symmetric algebra $S=P[X_1, \dots, X_n]$.  More precisely, the homogeneous strands look like
\[
\cdots \to 
{\Lambda}^{a+1}\otimes S_{b-1}
\xra{\kappa_{a+1,b-1}} 
{\Lambda}^{a}\otimes S_{b}
\xra{\kappa_{a,b}} 
{\Lambda}^{a-1}\otimes S_{b+1}
\to \cdots
\]
where $a+b=n$, $\Lambda$ denotes the exterior $P$-algebra $\Lambda F$ on the free $P$-module $F$ of rank $d$, 
and the differential is given by 
\[
\kappa_{a,b}
\left(
X_{i_1} \wedge\cdots\wedge X_{i_a} \otimes g
\right)
= 
\sum_{j=1}^a (-1)^j X_{i_1} \wedge\cdots\wedge \widehat{X}_{i_j}\wedge\dots\wedge X_{i_a}\otimes 
\left( 
X_{i_j} g
\right).
\]
The $P$-dual of this complex is isomorphic to the Koszul complex $K(X_1, \dots, X_n;D)$ for the variables on the $S$-module $D$. Its homogeneous strands are of the form
\[
\cdots \to 
{\Lambda}^{n-a+1}\otimes D_{b+1}
\xra{\eta_{n-a+1,b-1}} 
{\Lambda}^{n-a}\otimes D_{b}
\xra{\eta_{n-a,b}} 
{\Lambda}^{n-a-1}\otimes D_{b-1}
\to \cdots
\]

If  $a+b\ge 1$, these sequences are split exact 
and so all kernels (equivalently, images) are free $P$-modules; this follows from the fact that $X_1, \dots, X_n$ is a regular sequence on $S$ (see also \cite[Prop 2.2]{BuchEis-75}). 
With this in mind, define 
\begin{align*}
L_{a,b}&=\ker \kappa_{a,b}= \im \kappa_{a+1,b-1}
\subseteq \Lambda^a \otimes S_b 
\\
K_{a,b}&=\ker
\eta_{a,b}=\im \eta_{a+1,b+1}
\subseteq \Lambda^a \otimes D_b
\end{align*}
These are Schur and Weyl modules corresponding to hooks.
Define complexes
\begin{align*}
\mathbb L_m &\colon  0\to L_{n-1,m}\xra{\operatorname{kos}\otimes 1}
L_{n-2,m}\xra{\operatorname{kos}\otimes 1} \dots \to
L_{0,m} \xra{\varepsilon} P \to 0\\[2mm]
\mathbb K_m& \colon  0\to P \xra{\varepsilon^*}
K_{n-1,m} \xra{\operatorname{kos}\otimes 1} 
K_{n-2,m} \xra{\operatorname{kos}\otimes 1} \cdots \to
K_{0,m} \to 0
\end{align*}
where $\operatorname{kos}$ is the restriction of the differential on the Koszul complex 
\[
\Lambda^\bullet = K(x_1, \dots, x_n;P),
\]  
$\varepsilon$ is (induced by) the map
\[
\varepsilon \colon L_{0,m} = S = P[X_1, \dots, X_n] 
\to S(V) = k[x_1, \dots, x_n]=P
\]
that sends $X_i$ to $x_i$ for each $i$, and $\varepsilon^*$ is its dual.
\end{dfn}

\begin{ipg}\label{blk-be}
In \cite{BuchEis-75} Buchsbaum and Eisenbud show that  $\mathbb L_m$ is the minimal $P$-free resolution of $R/\mathfrak m^m$. Taking duals, one sees that 
\begin{equation}
\label{eqn-dual}
\Hom(\mathbb L_m,P) \cong \mathbb K_{m-1} 
\end{equation}
is the minimal $P$-free resolution of $\operatorname{Ext}^n_P(P/\mathfrak m^m,P)$.
\end{ipg}

Next we recall the new ingredients introduced in \cite{MiRa-18} that are needed for our resolutions of noncompressed algebras.

\begin{dfn} 
\label{dfn-K'}
Recall that
\[
D'_{s-m} = \im \, \left( S_m \xra{\Phi_m} D_{s-m} \right).
\]
One can verify that $\eta\circ(1\otimes \Phi)=(1\otimes \Phi)\circ\kappa$, and so 
the restrictions $\eta'$ of the maps $\eta$  give the subcomplex 
\begin{equation}
\label{def-eta'}
\cdots \xra{} 
{\Lambda}^i \otimes D_{s-m}' 
\xra{\eta'}
{\Lambda}^{i-1}\otimes D_{s-m-1}'
\xra{} \cdots 
\end{equation}  
This complex is no longer necessarily exact. We denote the cycles by 
\begin{equation}
\label{def-K'}
K_{i,s-m}' 
= \ker \left(
{\Lambda}^i \otimes D_{s-m}' 
\xra{\eta'} 
{\Lambda}^{i-1} \otimes D_{s-m-1}'
\right), 
\end{equation}
which is a free summand of $K_{i,s-m}$ by the proof of Proposition~\ref{prp-K'free}. 
We also denote the cokernels  by
\begin{equation}
\label{def-coker}
C_{i,s-m} 
=\coker \left( 
\Lambda^{i+1} \otimes D'_{s-m+1}
\xra{\eta'} 
K'_{i,s-m} 
\right).
\end{equation}
Notice that these cokernels are in fact the homologies of the subcomplex (\ref{def-eta'}), i.e. 
\begin{equation}
\label{C=homology}
C_{i,s-m} 
=\h_i \left(
\cdots 
\to 
\Lambda^{i+1} \otimes D'_{s-m+1}
\xra{\eta'} 
{\Lambda}^i \otimes D_{s-m}' 
\xra{\eta'}  
\cdots 
\right).
\end{equation}
Furthermore, they are always free, as we show in Proposition~\ref{prp-K'free}. 
One can verify that $\eta'\circ(\operatorname{kos} \otimes 1)=(\operatorname{kos} \otimes 1)\circ\eta'$, and so  
these modules fit together to make a subcomplex 
\[
\mathbb K_{s-m}' \colon  0\to P \xra{\pi\circ\varepsilon^*}
K_{n-1,s-m}' \xra{\operatorname{kos}\otimes 1} 
K_{n-2,s-m}' \xra{\operatorname{kos}\otimes 1} \cdots \to
K_{0,s-m}' \to 0
\]
of the complex $\mathbb K_{s-m}$, 
where $\pi\colon K_{n-1,m} \to K_{n-1,m}'$ is  any chosen projection to the summand. 

Next note that for every $i$ the commutative diagram on the right below induces the map on kernels indicated on the left 
\begin{equation*}
      \xymatrix{
     L_{i,m} \ar[d]^{1\otimes \Phi_{m} } \ar[r] & \Lambda^{i}\otimes S_m \ar[r]^\kappa \ar[d]^{1\otimes \Phi_{m}} & \Lambda^{i-1}\otimes S_{m+1} \ar[d]^{1\otimes \Phi_{m+1} }\\
     K'_{i,s-m} \ar[r] & \Lambda^{i}\otimes D'_{s-m} \ar[r]^{\eta'} & \Lambda^{i-1}\otimes D'_{s-m-1}  
    }
  \end{equation*}
which we also denote by $1\otimes \Phi_{m}$; unlike the other vertical maps, this induced map is not necessarily surjective. Note that the maps $1 \otimes \Phi$ and $\operatorname{kos} \otimes 1$ commute and hence the induced map 
\[
\mathbb L_m \xra{1 \otimes \Phi_m} 
\mathbb K_{s-m}' \subseteq 
\mathbb K_{s-m}
\]
is a chain map.
\end{dfn} 
In Section \ref{sec:socle} we show that in certain cases the modules $C_{i,s-m}$ are zero for $i=0,1$; see Proposition \ref{cokernelsvanish}. In general, one can see that they are all free modules.
\begin{prp}\label{prp-K'free}
The $P$-module $K_{i,s-m}'$ and $C_{i,s-m}$ defined above are  free for every $i$ and $m$. 
\end{prp}

\begin{proof} 
First we show that the image of each map $\eta'$ is a summand of $\Lambda^i \otimes D'_{s-m}$. For every $i$ we have a commutative diagram
      \begin{equation*}
      \xymatrix{
      \Lambda^{i+1}\otimes S_{m-1} \ar[r]^{ \ \ \ \ \ \kappa} \ar[d]^{1\otimes \Phi_{m-1}} & L_{i,m} \ar[d]^{1\otimes \Phi_{m} }\\
      \Lambda^{i+1}\otimes D'_{s-m-1} \ar[r]^{ \ \ \ \ \ \ \eta'} & \Lambda^{i} \otimes D'_{s-m} 
    }
  \end{equation*}
Since $\kappa$ and $1 \otimes \Phi_{m-1}$ are surjective, the maps $\eta'$ and $1 \otimes \Phi_{m}$ have the same image. Furthermore, we know that $1 \otimes \Phi_{m}$ is represented by a matrix whose entries come from the coefficients of $\Phi$, hence are units or equal to 0, and so its image is a summand of $\Lambda^i \otimes D'_{s-m}$. Therefore the image of $\eta'$ is also a summand of $\Lambda^i \otimes D'_{s-m}$ and so is free for all $i,m$. 

From the short exact sequences 
\[
0 \lra K_{i,s-m}' \lra \Lambda^i \otimes D_{s-m}' \xra{\eta'} \im \eta' \lra 0
\]
and 
\[
0 \lra \im \eta' \lra K_{i,s-m}' \lra C_{i,s-m}  \lra 0
\]
we now conclude that $K_{i,s-m}'$ and $C_{i,s-m}$ are free as well. 
\end{proof} 

With these ingredients, we now obtain a resolution of any standard graded Gorenstein algebra. In these resolutions the only nonexplicit modules and maps are coming from the cokernels $C$ defined above. 
The proof is modeled on that of the main result in \cite{MiRa-18}. 

We note that, when the algebra is compressed, these cokernels always vanish by the remarks in Definitions~\ref{dfn-LK} and \ref{dfn-K'} since in that case $D=D'$. 
When the algebra is not compressed they can be nonzero. However, further below, we show that under certain assumptions some of these cokernels will vanish.

We first recall the proposition from \cite{MiRa-18} that is used in the proof. 
For the explicit differentials in the resulting complex, see \cite[Prop 2.1]{MiRa-18}. 

\begin{prp}
\label{prp-min}
Let $R$ be a ring, and let $(X,\partial^X)$ and $(Y,\partial^Y)$ be complexes of $R$-modules, and consider a chain map $\phi\colon X \to Y$.
    For each $i$, let 
    \[
    A_i=\Ker \phi_i, ~
    B_i=\Im \phi_i, ~
    {\textrm{ and }}~
    C_i=\coker \phi_i
    \] 
    Suppose that, for every $i>0$, $A_i$ and $B_i$ are direct summands of $X_i$ and $Y_i$, respectively, which we display visually with choices of splitting maps as 
      \begin{equation*}
      \xymatrix{
      &A_3 \ar[d]^{}
      &A_2 \ar[d]^{}
      &A_1 \ar[d]^{} 
      & 
      & 
      &   \\
      \cdots  \ar[r]
      &X_3 \ar[r] \ar[d]\ar@/^1pc/[u]^{}
      &X_2 \ar[r]\ar[d]\ar@/^1pc/[u]^{}
      &X_1 \ar[r]\ar[d] \ar@/^1pc/[u]^{} 
      &X_0 \ar[r]\ar[d]
      &X_{-1} \ar[r]\ar[d]
      & \cdots  \\
      \cdots  \ar[r]
      &Y_3 \ar[r] \ar[d]^{}
      &Y_2 \ar[r] \ar[d]^{}
      &Y_1 \ar[r] \ar[d]^{} 
      &Y_0 \ar[r]
      &Y_{-1} \ar[r]
      & \cdots  \\
      &C_3 \ar@/^1pc/[u]^{}
      &C_2 \ar@/^1pc/[u]^{}
      &C_1 \ar@/^1pc/[u]^{}
      & 
      & 
      &   
    }
  \end{equation*}
Then the mapping cone of $\phi$ is homotopy equivalent to a complex of the form 
      \begin{equation*}      \xymatrixrowsep{0.2pc} 
      \xymatrix{
     \cdots \ar[r]&A_3 \ar[r]^{} 
      &A_2 \ar[r]^{} 
      &A_1 \ar[r]^{} 
      &X_0 \ar[r]^{} 
      &X_{-1} \ar[r]^{} 
      & \cdots  \\
 \cdots     
      & \oplus
      &\oplus
      &\oplus
      &\oplus
      &\oplus
      &         \\
      \cdots\ar[uur]^{}  \ar[r]^{} 
      &C_4 \ar[r]^{} \ar[uur]
      &C_3 \ar[r]^{} \ar[uur]
      &C_2 \ar[r]^{} \ar[uur]
      &C_1 \ar[r]^{} 
      &Y_0 \ar[r]^{} 
      & \cdots 
    }
  \end{equation*}
where the horizontal maps are those induced by the differentials $\partial^X$ and $\partial^Y$ and the diagonal maps are a composition of maps in the following order: the splitting from $C_{i+2}$ to $Y_{i+2}$, the differential $\partial^Y\colon Y_{i+2} \to Y_{i+1}$, the projection to the image $B_{i+1}$ of $\phi_{i+1}$, the splitting of $\phi_{i+1}$ from $B_{i+1}$ up to $X_{i+1}$, the differential $\partial^X\colon X_{i+1} \to X_{i}$, and for $i>0$ the projection to the summand $A_i$. 
\end{prp}

In other words, for $i>0$ the $X_i$ have been replaced by the kernels $A_i$ and the $Y_i$ have been replaced by the cokernels $C_i$ at the cost of some induced snake-lemma-like maps $C_{i+2} \lra A_{i}$ (or $C_{2} \lra X_0$ if $i=0$).

\begin{thm}
\label{resn-general}
For any standard graded Artinian Gorenstein algebra $P/J$ and every integer $m\ge 1$, 
the mapping cone of the following chain map is a $P$-free resolution of the quotient $P/J_{\ge m}$ of $P$ by the truncated ideal $J_{\ge m}$.
\begin{equation}
    \label{resolution}
\begin{split} 
\xymatrixrowsep{0.6pc} \xymatrixcolsep{0.7pc}
    \xymatrix{  
        &0\ar[dd]& &\\  
    P\ar[dd]_{\tiny\begin{bmatrix} \pi \circ \varepsilon^* \\ *\end{bmatrix}}&& &\\    
     & L_{n-1,m} 
      \ar[dl]_{\hspace{3mm}(1\otimes \Phi_{m}, 0)}\ar[dd]^{\partial^{\mathbb L}}&&\\
    K'_{n-1,s-m}\oplus W_{n-1,s-m-1}\ar[dd]_{\tiny\begin{bmatrix} \partial^{\mathbb K} & * \\ 0 & \partial^{\mathbb K} \end{bmatrix}} 
      &\ar[d]&&\\
&{}^{\vdots}\ar[dd]^{\partial^{\mathbb L}}&&\\ 
{}^{\vdots}\ar[dd]_{\tiny\begin{bmatrix} \partial^{\mathbb K} & * \\ 0 & \partial^{\mathbb K} \end{bmatrix}}&&&\\ 
       & L_{2,m}\ar[dl]_{\hspace{3mm}(1\otimes \Phi_{m}, 0)}\ar[dd]^{\partial^{\mathbb L}}
      &&\\
      K'_{2,s-m} \oplus W_{2,s-m-1}\ar[dd]_{\tiny\begin{bmatrix} \partial^{\mathbb K} & * \\ 0 & \partial^{\mathbb K} \end{bmatrix}} 
      &&&\\
      &  L_{1,m} \ar[dl]_{\hspace{3mm}(1\otimes \Phi_{m}, 0)}\ar[dd]^{\partial^{\mathbb L}}
      &&\\    
     K'_{1,s-m}\oplus W_{1,s-m-1}\ar[dd]_{\tiny\begin{bmatrix} \partial^{\mathbb K} & * \\ 0 & \partial^{\mathbb K} \end{bmatrix}} 
      &&&\\    
      & L_{0,m}\ar[dd]^{\varepsilon}\ar[dl]_{\hspace{3mm}(1\otimes \Phi_{m}, 0)}
      &&\\   
     K'_{0,s-m}\oplus W_{0,s-m-1}
      &&&\\
      &P&&
    }
\end{split}
\end{equation}
where $W_{a,b}=\oplus_{j=0}^{b} C_{a,j}$ with  $C_{a,j}=\coker\left( \Lambda^{a+1} \otimes D_{j+1}' \xra{\eta'} K_{a,j}'\right)$, as defined in Definition \ref{dfn-K'}, and the vertical differentials along the complexes are induced by the differentials of $\mathbb L$ and $\mathbb K$ as shown 
with the exception of the map $*$. 
\end{thm} 

\begin{proof} 
We use a (descending) inductive process to go from a resolution of $P/J_{\ge m+1}$ to a resolution of $P/J_{\ge m}$ as follows. Consider the short exact sequence
  \begin{equation*}
    0 \to J_{\ge m}/J_{\ge m+1} \to P/J_{\ge m+1} \to P/J_{\ge m}
    \to 0
  \end{equation*}
and let $F$ be a graded free resolution of $P/J_{\ge m+1}$. Then $J_{\ge m}/J_{\ge m+1}$ is the graded component $J_{m}$, which is a $k$-vector space, say of dimension $t$, and hence resolved by $t$ copies $E^t$ of the Koszul complex $E=K(x_1,\dots , x_n;P)$ on the variables. 
So for any lift 
\[
f\colon E^t \to F
\] 
of the inclusion $J_{m} \cong J_{\ge m}/J_{\ge m+1} \hookrightarrow P/J_{\ge m+1}$, the mapping cone of $f$ is a free resolution of $P/J_{\ge m}$. 
In order to get an explicit lift, we will use, in place of $E^t$, a mapping cone of Koszul complexes that is quasi-isomorphic to $E^t$, namely $E \otimes S_m \xra{1\otimes \Phi_m} E \otimes D_{s-m}'$: 
indeed since $D_{s-m}'$ is a free $P$-module, the surjection $S_m \xra{\Phi_m} D_{s-m}'$ is split and so the map $E \otimes S_m \xra{1\otimes \Phi_m} E \otimes D_{s-m}'$ is also a split surjection. Hence it is homotopy equivalent (in particular, quasi-isomorphic) to its kernel, which is the tensor product of $E$ with the kernel of $S_m \xra{\Phi_m} D_{s-m}’$, which is a free $P$-module of rank $t$ since its basis maps bijectively via $\varepsilon$ to the minimal generators of $J_m$ (see (\ref{pairings})). 

For the base of the inductive process we begin with a resolution of $P/J_{\ge s+1}$. 
Note that in degrees $s+1$ and higher, the ideal is equal to a power of the homogeneous maximal ideal, more precisely, $J_{\ge s+1} =P_{\ge s+1}=\fm^{s+1}$, since 
\[
J_{s+1} = \varepsilon \left( 
\ker 
(S_{s+1} \xra{\Phi_{s+1}} D_{-1}')
\right)
= \varepsilon 
\left( 
S_{s+1}
\right)
=P_{s+1}
\]
Hence a free resolution of $P/J_{\ge s+1}$ is given by the Buchsbaum-Eisenbud complex $\mathbb L_{s+1}$; see Definition~\ref{dfn-LK}. 
To proceed with the induction, we can think of the base case resolution $\mathbb L_{s+1}$ as the mapping cone of $1 \otimes \Phi_{s+1} \colon \mathbb L_{s+1} \to \mathbb K_{-1}'$ since $\mathbb K_{-1}'=0$. We note that when performing inductive steps the very first step from the base case, namely $P/J_{\geq s+1}$ to $P/J_{\geq s}$, has a slightly different explanation since $\mathbb K_{-1}'=0$. We will explain it separately at the end. 

Assume, by induction, that the resolution of $P/J_{\ge m+1}$ is as given in the statement and form the following diagram (\ref{cube1}), in which the mapping cone of the right wall is the aforementioned resolution of $P/J_{\ge m+1}$ and the mapping cone of the left wall is a free resolution of $J_m=J_{\ge m}/J_{\ge m+1}$, since it is homotopy equivalent to  $E^t$, where $t=\dim_k J_{m}$ as explained further above.



\begin{equation}
    \label{cube1}
\begin{split} 
\xymatrixrowsep{0.6pc} \xymatrixcolsep{0.7pc}
    \xymatrix{   
      && P\ar[dd]&\\
      & \Lambda^n \otimes S_{m} \ar[rr]^{\hspace{-15mm}\kappa}
      \ar[dl]_{1\otimes \Phi_{m}}\ar[d]&&L_{{n-1},m+1}\ar[dl]_{1\otimes \Phi_{m+1}}\ar[d]\\
     \Lambda^n\otimes D_{s-m}'\ar[rr]^{\hspace{-1mm}(\eta',0)}\ar[d] 
      &{}^{\vdots}\ar[d]&K_{{n-1},{s-m-1}}'\oplus W_{n-1,s-m-2}\ar[d]&{}^{\vdots}\ar[d]\\
      {}^{\vdots}\ar[d] & \Lambda^2 \otimes S_{m} \ar[rr]^{\hspace{-15mm}\kappa}\ar[dl]_{1\otimes \Phi_{m}}\ar[dd]
      &{}^{\vdots}\ar[d]&L_{{1},m+1}\ar[dl]_{1\otimes \Phi{m+1}}\ar[dd]\\
      \Lambda^2\otimes D_{s-m}'\ar[rr]^{\hspace{-1mm}(\eta',0)}\ar[dd]
      &&K_{{1},{s-m-1}}'\oplus W_{1,s-m-2}\ar[dd]&\\
      & \Lambda^1 \otimes S_{m} \ar[rr]^{\hspace{-15mm}\kappa}\ar[dl]_{1\otimes \Phi_{m}}\ar[dd]
      &&L_{{0},m+1}\ar[dl]_{1\otimes \Phi_{m+1}}\ar[dd]\\    
      \Lambda^1\otimes D_{s-m}'\ar[rr]^{\hspace{-1mm}(\eta',0)}\ar[dd]
      &&K_{{0},{s-m-1}}'\oplus W_{0,s-m-2}\ar[dd]&\\    
      & \Lambda^0 \otimes S_{m} \ar[rr]^{\hspace{-15mm}\varepsilon}\ar[dl]_{1\otimes \Phi_{m}}
      &&P\ar[dl]\\   
     \Lambda^0\otimes D_{{s-m}}'\ar[rr]
      &&0&
    }
\end{split}
   \end{equation}
Here $W_{a,b}=\oplus_{j=0}^{b} C_{a,j}$ and  $C_{a,j}=\coker( \Lambda^{a+1} \otimes D_{j+1}' \xra{\eta'} K_{a,j}')$.
One can easily check that the induced map on the 0th homologies agrees
  with the inclusion of $J_{\ge m}/J_{\ge m+1}$ into
  $P/J_{\ge m+1}$. Therefore the total
  complex of the entire diagram is a resolution of
  $P/J_{\ge m}$. 

The final step of the inductive process is to simplify this diagram. 
Indeed, Proposition~\ref{prp-min} enables us to simplify the front and back walls of diagram (\ref{cube1}). Assume first that $m>s$ so that the front right hand column containing $\mathbb K'$ is not zero. 
Along the back wall, the maps $\kappa$ are surjective and have kernels equal to $L_{a,m}$. These maps are clearly split as each $L$ is free. 
Along the front wall, the kernels of the horizontal maps $\eta'$ are $K_{a,s-m}'$ and the cokernels are equal to 
\[
\left( \frac{K_{a,s-m-1}'}{\im(\eta')} \right) \oplus W_{a,s-m-2} 
=
C_{a,s-m-1} \oplus W_{a,s-m-2}  
=
W_{a,s-m-1}.
\]
Note further that $K_{n,j}'$, $L_{n,j}$, and $C_{n,j}$ vanish for all $j$. 
The horizontal maps are split by Proposition~\ref{prp-K'free} (although the map $\varepsilon$ situated in homological degree 0 is not split). 
Putting these together,  Proposition~\ref{prp-min} yields that the totalization of diagram (\ref{cube1}) is homotopy equivalent to the totalization of the following diagram, where one can calculate that the map from $P$ to $K'_{n-1,s-m}$ is  the map $\pi \circ \varepsilon^*$ in the complex $\mathbb K_{s-m}$.  
\begin{equation}
    \label{cube2}
\begin{split} 
\xymatrixrowsep{0.6pc} \xymatrixcolsep{0.7pc}
    \xymatrix{  
        &0\ar[dd]& &\\  
    P\ar[dd]_{\tiny\begin{bmatrix} \pi \circ \varepsilon^* \\ *\end{bmatrix}}&& &\\    
     & L_{n-1,m} 
      \ar[dl]_{\hspace{3mm}(1\otimes \Phi_{m}, 0)}\ar[dd]^{\partial^{\mathbb L}}&&\\
    K'_{n-1,s-m}\oplus W_{n-1,s-m-1}
  \ar[dd]_{\tiny\begin{bmatrix} \partial^{\mathbb K} & * \\ 0 & \partial^{\mathbb K} \end{bmatrix}} 
      &\ar[d]&&\\
&{}^{\vdots}\ar[dd]^{\partial^{\mathbb L}}&&\\ 
{}^{\vdots}\ar[dd]_{\tiny\begin{bmatrix} \partial^{\mathbb K} & * \\ 0 & \partial^{\mathbb K} \end{bmatrix}}&&&\\ 
       & L_{2,m}\ar[dl]_{\hspace{3mm}(1\otimes \Phi_{m}, 0)}\ar[dd]^ {\partial^{\mathbb L}}
      &&\\
      K'_{2,s-m} \oplus W_{2,s-m-1}
      \ar[dd]_{\tiny\begin{bmatrix} \partial^{\mathbb K} & * \\ 0 & \partial^{\mathbb K} \end{bmatrix}} 
      &&&\\
      &  L_{1,m} \ar[dl]_{\hspace{3mm}(1\otimes \Phi_{m}, 0)}\ar[dd]^ {\partial^{\mathbb L}}
      &&\\    
     K'_{1,s-m}\oplus W_{1,s-m-1}
     \ar[dd]_{\tiny\begin{bmatrix} \partial^{\mathbb K} & * \\ 0 & \partial^{\mathbb K} \end{bmatrix}} 
      &&&\\    
      & L_{0,m}\ar[dd]^{\varepsilon}\ar[dl]_{\hspace{3mm}(1\otimes \Phi_{m}, 0)}
      &&\\   
     K'_{0,s-m}\oplus W_{0,s-m-1}
      &&&\\
      &P&&
    }
\end{split}
   \end{equation}
In the case that $m=s$ (the first inductive step from the base case $m=s+1$), the front right hand column of (\ref{cube1}) is zero, and we can simplify the diagram the same way as before along the back wall (the maps $\kappa$ are surjective and have kernels equal to $L_{a,m}$). 
Along the front wall the kernels are simply equal to the left hand front column $\Lambda^i\otimes D'_0$, which is in fact the complex $\mathbb K'_{0}$. (Note that $\Lambda^n\otimes D'_0$ has rank 1 and therefore can be identified with $P$ and also that $W_{i,s-m-1}=W_{i,-1}=0$.) 
This completes the inductive step and hence the proof. 
\end{proof}

To conclude this section we give the practical consequences of Theorem~\ref{resn-general}. 
In particular, we analyze the degrees of the maps in the resolution and then discuss what can be deduced about the {\it minimal} resolution. 
This information will be applied as the main tool in the proof of Corollary~\ref{generatingdegreesofJ}.

\begin{rmk}
\label{rmk-resn-minimal}
For convenience, we rewrite the mapping cone from Theorem~\ref{resn-general} horizontally as follows, where we have set $t=s-m$,   
\begin{align*}
\xymatrixcolsep{1.3pc}
    \xymatrix{
      &L_{n-1,m} \ar[r] \ar[d]^{
     (1 \otimes \Phi,0)}
      &\cdots \ar[r]
      &L_{1,m} \ar[r]\ar[d]^{(1 \otimes \Phi,0)}
      &L_{0,m} \ar[r]^{\varepsilon}\ar[d]^{(1 \otimes \Phi,0)}
      &P\\
P\ar[r]^{\varepsilon^* \hspace{12mm} \ }
&K_{n-1,t}'\oplus W_{n-1,t-1}\ar[r] 
&\cdots \ar[r]
&K_{1,t}'\oplus W_{1,t-1} \ar[r]
&K_{0,t}'\oplus W_{0,t-1} 
&
  }
  \end{align*}
and note that in fact one always has $W_{0,t-1}=0$ by Proposition~\ref{cokernelsvanish}. 

First, let us understand the maps 
\[
\begin{bmatrix} \partial^{\mathbb K} & * \\ 0 & \partial^{\mathbb K} \end{bmatrix}
\]
in the second row of the diagram above better.  For each $i$, recall that $W_{i,t-1}=\oplus_{j=0}^{t-1} C_{i,j}$ and that by Proposition~\ref{prp-K'free} each free module $C_{i,j}$ splits out of $K'_{i,j}$ which splits out of $\Lambda^i\otimes D_j$. 
So, $C_{i,j}$ and $K'_{i,j}$ have bases consisting of elements that are $k$-linear combinations of elements of the form $e \otimes g$ where $e=e_{k_1} \wedge \cdots \wedge e_{k_i} $ is an exterior monomial of degree $i$ and $g$ is a divided power monomial of degree $j$. 
The maps $\partial^{\mathbb K} \colon K_{i,t}' \to K_{i-1,t}'$ and $\partial^{\mathbb K} \colon W_{i,t-1} \to W_{i-1,t-1}$ take each such basis element to a sum of elements of the form $r e'\otimes g'$ where $r\in \fm\backslash\fm^2$ (i.e., $r$ is homogeneous of degree $1$), $e'$ is an exterior monomial of degree $i-1$, and $g'$ is a divided power monomial of degree $t$. That is, these two maps are linear.

The map $*\colon W_{i,t-1} \to K'_{i-1,t}$ is more complicated. 
If one follows carefully each application of Proposition~\ref{prp-min} in the proof of Theorem~\ref{resn-general}, one sees that the map $*$ is a composition of the induced maps from cokernels to kernels, that is, the diagonal maps in the conclusion of the proposition. 
Recall that these maps are a composition of maps (ignoring the various projections that don't change the degrees of the coefficients) starting with the  differential on $K'_{i,t}\oplus W_{i,t-1}$ (which we can assume we understand by induction on $t$) followed by the splitting of $(\eta',0)$, and last by the map $\operatorname{kos}\otimes 1$. 
So a basis element of each summand $C_{i,j}$ for $j=0, \dots, t-1$ maps under $*$ to an element of $K'_{i-1,t}$ that is a sum of elements of the form $r e'\otimes g'$ where $r\in \fm^{t-j+1} \backslash\fm^{t-j+2}$ (i.e., $r$ is homogeneous of degree $t-j+1$), $e'$ is an exterior monomial of degree $i-1$, and $g'$ is a divided power monomial of degree $t$. 
Altogether, taking $j=0,\dots,t-1$, we see that the basis of $W_{1,t-1}$ maps to sums of elements of the form $r e' \otimes g'$ where 
\begin{itemize}
\item 
$r\in \fm^{j'} \backslash\fm^{j'+1}$ with $2 \leq j' \leq t+1$ (i.e., $r$ is homogeneous of degree $j'$), 
\item 
$e'$ is an exterior monomial of degree $i-1$, and 
\item 
$g'$ is a divided power monomial of degree $t$. 
\end{itemize} 

Next, we apply Proposition~\ref{prp-min} once again, this time to reduce the resolution in the diagram above to a minimal one. 
Indeed, since the coefficients of $\Phi$ are units, one can show that the vertical maps $(1 \otimes \Phi,0)$ are completely split (the kernels and images the maps $(1 \otimes \Phi,0)$ split out of their  domains and codomains, respectively). 
Although reducing the entire resolution in general is too complicated without more knowledge of these kernels and cokernels, we can say something about the degrees of the maps that result.  

In particular, we analyze the beginning of the resolution in order to determine the degrees of generators of $J_{\geq m}$. 
Looking at the maps resulting from the proposition and considering the discussion above, one arrives at the following conclusion. 
Choose a splitting $\sigma$ of the surjective map 
$L_{0,m} \xra{1 \otimes \Phi} K_{0,t}'$. 

\begin{cor}
\label{cor-degrees}
For any standard graded Gorenstein algebra $P/J$, 
a homogeneous generating set of $J_{\geq m}$ is given by a union of the following sets where $t=s-m$:  
\begin{itemize}
\item 
the image under the map  $\varepsilon$ of a basis of \[
\ker (1 \otimes \Phi \colon L_{0,m} \to K_{0,t}'),
\] 
\item 
the image under the composition $\varepsilon \circ \sigma \circ \partial^{\mathbb K}$ of (a lift to $K_{1,t}'$ of) a basis of 
\[
\coker (1 \otimes \Phi \colon L_{1,m} \to K_{1,t}'), 
\]
\item 
and the image under the composition $\varepsilon \circ \sigma \circ *$ of a basis of $W_{1,t-1}$.

\end{itemize} 

Their respective degrees are
\begin{itemize}
\item $m$
\item $m+1$
\item $m+j'$ with $2 \leq j' \leq t+1$
\end{itemize}
\end{cor}

In our main application to Frobenius powers in the next section, the generators of the ideal $J$ in degrees less than or equal to $\frac{s}{2}$ are known and to obtain the generators in higher degrees we prove that $W_{1,t-1}=0$ for $m=\ceil{\frac{s}{2}}$ (see Corollary~\ref{generatingdegreesofJ}) so that we obtain all the generators of $J_{\geq m}$ from Corollary~\ref{cor-degrees}. 
\end{rmk} 

  \section{ Key technical results and consequence for socles}
  \label{sec:socle}

Let $k$ be a field of positive characteristic $p$ and let $P=k[x_1, \dots, x_n]$ be the polynomial ring with homogeneous maximal ideal
$({\bf x})=(x_1,\ldots,x_n)$. 
For any nonzero homogeneous polynomial $f\in P$ of degree $d$, consider 
the hypersurface ring
\[
R=P/(f)
\] 
with homogeneous maximal ideal $\fm=({\bf x})$, where we use the same notation for the variables and their images in $R$. 
We use the notation $q=p^e$ whenever it does not cause confusion. 

Our goal in this paper is to study the minimal free resolutions of the ideals
\[
\fm^{[q]}=({\bf x}^q)=(x_1^{q}, \dots, x_n^{q}) 
\subseteq R
\]
 for a large class of polynomials $f$ in $n=3$ variables, which we do in Section \ref{sec:frobenius}. 
In this section we prove the key technical result, Proposition~\ref{numberofgeneratorsofJ}, needed for that. 
Here, we work with any number of variables, but the strong stability results we find in the  three variable case are what drive the phenomena observed in Section~\ref{sec:frobenius}. 
The class of polynomials $f$ that we consider consists of those that are \qlinkcomp, that is, those whose links  $J=((x_1^q,\ldots,x_n^q):f)$ are $(x_1^q,\ldots,x_n^q)$-compressed; see Definition \ref{def-qlinkcomp}. 

As a corollary of Proposition~\ref{numberofgeneratorsofJ}, we get the following result about the dimension of the socle of $R/\fm^{[q]}$, 
which follows in view of Remark~\ref{muofJ}.
Note that in the statement 
\begin{itemize}
    \item $s$ is even when $p$ is odd and $d$ is even (or $p=2$ and $n-d$ is even),  
    \item $s$ is odd when $p$ is odd and $d$ is odd (or $p=2$ and $n-d$ is odd).
\end{itemize}

In view of Remark~\ref{rmk-general}, the result below holds for general choices of $f$ of degree $d$.

\begin{thm}
\label{soc-thm}
Let $R=k[x_1,\dots,x_n]/(f)$ be a standard graded hypersurface ring over a field $k$ with homogeneous maximal ideal $\fm$. 
Set 
\[
q=p^e
\qquad 
d= \deg f 
\ \ {\text{ and }} \ \  s=n(q-1)-d.
\]
Assume that $n\geq 3$ and $q \geq (n+d)/(n-2)$. 
Assume further that $f$ is \linkcomp{$q$}. 
Then the following hold for the socle module
$\soc(R/\fm^{[q]})$.  
\begin{itemize}
\item[(1)] 
Its generators lie in degrees
\[
s_i=\frac{1}{2}(n(q-1)+d-i)
\qquad 
i \in
\begin{cases}
\{2\}, 
& {\textrm{ if $s$ is even }} \\
\{1, 3\}, 
& {\textrm{ if $s$ is odd.}} 
\end{cases}
\]

\item[(2)] 
If $n=3$, its dimension satisfies 
\begin{align*}
&\qquad {\textrm{ if $s$ is even, }} \dim_k \soc \left(R/\fm^{[q]}\right)_{s_2}  \!\!\!= 2d 
\\
&\qquad {\textrm{ if $s$ is odd, }}
\dim_k \soc \left(R/\fm^{[q]}\right)_{s_1} \!\!\!= d
{\textrm{ \ and \ }}
\dim_k \soc \left(R/\fm^{[q]}\right)_{s_3}\!\!\!\leq 3d. 
\end{align*}

\item[(3)] 
If $n \geq 4$, its dimension is equal to a polynomial in $q$ with leading term
\[
\frac{2}{(n-2)!} \left[ \left(\frac{n}{2}\right)^{n-2}\!\!\!\!-n\left(\frac{n-2}{2}\right)^{n-2} \right] q^{n-2} 
\qquad 
{\textrm{ for $s$ even, }}
\]
and bounded by a polynomial in $q$ of degree $n-2$ for $s$ odd. 
Furthermore, the coefficients of these polynomials depend only on $n$ and $d$. 
\end{itemize}
\end{thm}

Our strategy in this section is as follows.
First we rewrite the quotient rings as
\[
R/\fm^{[q]} \cong P/I_q,
{\textrm{ where }}
I_q=({\bf x}^q)+(f).
\] 
As Kustin, Rahmati, and Vraciu do in \cite{KRV-12}, we consider the link of $I_q$ in $P$ 
\begin{equation}
J_q 
=\left( ({\bf x}^q) :_P I_q \right)
=\left( ({\bf x}^q) :_P f \right)
\end{equation}
via the complete intersection ideal $\fc=({\bf x}^q)$. 
Since $J_q$ is linked to the almost complete intersection ideal $I_q$, it follows that $J_q$ is a Gorenstein ideal, that is, $P/J_q$ is a Gorenstein ring. 

First we recall a useful connection between the socle degrees of $P/I_q$ and the degrees of the generators of its link $J_q$.
We will apply this remark to deduce both Theorem~\ref{soc-thm} and Corollary~\ref{KUsoc} from Proposition~\ref{numberofgeneratorsofJ}.

\begin{rmk}\label{muofJ}
Let $P=k[x_1, \dots, x_n]$ and let $f$ be a homogeneous polynomial  in $P$. Let   ${\bf x}= x_1,\dots, x_n$ and  $I_q=({\bf x}^q)+(f)=(x_1^q,\cdots,x_n^q,f)$. 

Assume that $x_1^q,\cdots,x_n^q$ are part of a minimal generating set for the Gorenstein ideal $J_q=({\bf x}^q):I_q$.
For any basis $\{u_1, \dots, u_{\ell}\}$ of the socle of $P/I_q$, there are homogeneous elements $v_1, \cdots, v_{\ell}$ in $P$ such that $\{x_1^q,\dots,x_n^q,v_1, \cdots, v_{\ell}\}$ minimally generate $J_q$ and $|u_i|+|v_i|=n(q-1)$. See, for example,  \cite[1.4]{KV-07} or \cite{MMN-11}. In particular, one has 
\begin{equation*}
    \mu(J_q)=\dim_k (\soc P/I_q)+n
\end{equation*}

Note that the converse also holds, that is, if  $\{x_1^q,\dots,x_n^q,v_1, \cdots, v_{\ell}\}$   is a minimal set of homogeneous generators for  $J_q$,  there are elements $u_1, \dots, u_{\ell}$ in $P/I_q$ with $|u_i|+|v_i|=n(q-1)$ that form a basis for the socle of $P/I_q$.
 \end{rmk}

For the rest of this section, we assume that $f$ is chosen so that $J_q$ is $(x_1^{q}, \dots, x_d^{q})$-compressed. 
Next, we work towards our goal of finding the number and degrees of minimal homogeneous generators of $J_q$. 
Recall from Lemma~\ref{phi} that the inverse polynomial of $J_q$ is 
\[
\varphi=f[x_1^{(q-1)} \cdots x_n^{(q-1)}].
\]
Its degree equals 
\[
\deg \varphi 
=s 
=n(q-1)-d, 
{\textrm{ \ \ where \ \ }}
d=\deg f.
\]

Our strategy 
is as follows. 
Recall that, since $J_q$ is $({\bf x}^q)$-compressed, its minimal homogeneous generators besides those of $({\bf x}^q)$ lie in degrees greater than or equal to $\frac{s}{2}$; see the remark in Definition~\ref{relcompressed} or Lemma~\ref{TFAErelativelycompressed}(3). 
For the ideal $J_q$, we figure out the number and degrees of these additional generators by determining the beginning of the resolution of $P/(J_q)_{\ge \frac{s}{2}}$ via Proposition~\ref{resn-general}. In this case, the unknown cokernels in the beginning part of the resolution turn out to be zero by the following result. 
We note that we cannot find a full explicit resolution of $P/J_q$ since the appropriate cokernels no longer vanish in degrees below $\frac{s}{2}$.

We now show that when $m=\lceil\frac{s}{2}\rceil$ 
the modules in Diagram~\ref{resolution}
\[
W_{i,s-m-1}=\oplus_{j=0}^{s-m-1} C_{i,j}
\] 
vanish for $i=0,1$. Recall from Definition~\ref{dfn-K'} that
\begin{equation*}
C_{i,s-\ell-1} 
=\coker \left( 
\Lambda^{i+1} \otimes D'_{s-\ell}
\xra{\eta'} 
K'_{i,s-\ell-1} 
\right)
\end{equation*}

\begin{prp}
\label{cokernelsvanish}
Assume the notation above. 
For all $\ell\le s-1$, one has 
\[
C_{0,s-\ell-1}=0.
\]
Moreover, if  $J_q$ is  $(x_1^{q}, \dots, x_n^{q})$-compressed, then 
\[
C_{1,s-\ell-1}=0
\]
for every $\ell \geq s/2-1$.
 
\end{prp}

\begin{proof}
We begin with $C_{0,s-\ell-1}$. 
It is easy to see that 
the map 
\[
\Lambda^1 \otimes D'_{s-\ell} \xra{\eta'} \Lambda^0 \otimes D'_{s-\ell-1}=K_{0,s-\ell-1}'
\]
is surjective. Indeed, since by definition $D'_{s-\ell-1}=\Im \Phi_{\ell+1}$, it is enough to show that, for any monomial $G$ in $S_{\ell+1}$, the product $G\Phi$ is in the image of $\eta'$. To see this, suppose that $X_i$ is a variable appearing in the monomial $G$. 
Then one has 
\[
\eta'\left( \sum e_i \otimes G{X_i}\varphi  \right) 
=X_i (G{X_i} \Phi)
=(X_i G{X_i}) \Phi
=G \Phi
\]
where the last equality is because $D$ is an $S$-module. 

To prove the second assertion, 
we first calculate the differential on each basis element of $\Lambda^1 \otimes D_{s-\ell-1}'$ as follows. 
First recall from Proposition~\ref{prp-D'basis} that a basis for $D_{s-\ell-1}'$ is given by the monomials $X_1^{(a_1)}\cdots X_n^{(a_n)}$ of degree $s-\ell-1$ with each exponent satisfying $0\leq a_k < q$. 
For each $i=1,\dots,n$ and each such monomial we have 
\[
\eta' ( e_i \otimes X_1^{(a_1)}\cdots X_n^{(a_n)} ) 
=
\begin{cases}
1\otimes X_1^{(a_1)}\cdots X_i^{(a_i-1)} \cdots X_n^{(a_n)}
& {\text{ \ if \ }} a_i \ge 1 
\\
0
& {\text{ \ if \ }} a_i=0 \end{cases}
\]
Since the elements $1\otimes X_1^{(b_1)}\cdots X_n^{(b_n)}$ with $0\leq b_i \leq q-1$ form a basis of $\Lambda^0 \otimes D_{s-\ell-2}'$, the kernel $K_{s-\ell-1}'$ is clearly generated by the cycles 
\[
\{ 
e_i \otimes X_1^{(a_1)}\cdots X_i^{(0)} \cdots X_n^{(a_n)} 
\}
\]
for $1\leq i\leq d$ and $0\leq a_k\leq q-1$
and the cycles 
\[
\{
e_i \otimes X_1^{(a_1)}\cdots X_i^{(a_i-1)}\cdots X_j^{(a_j)} \cdots X_n^{(a_n)} 
- 
e_j \otimes X_1^{(a_1)}\cdots X_i^{(a_i)} \cdots X_j^{(a_j-1)} \cdots X_n^{(a_n)}
\}
\]
for $1 \leq i < j \leq d$ and $1\leq a_k\leq q-1$. 
The latter of these cycles is the image of 
\[
e_i \wedge e_j  \otimes X_1^{(a_1)}\cdots X_i^{(a_i)} \cdots X_j^{(a_j)} \cdots X_n^{(a_n)}
\] 
under $\eta'$. 
For the former cycles, note that $\sum a_j \leq d(q-1)/2 $  implies that $a_j < q-1$ for some $j \not = i$. 
Assume without loss of generality that $j<i$ (else one just needs an additional minus sign). Then one has 
\[
e_i \otimes X_1^{(a_1)}\cdots X_i^{(0)} \cdots X_n^{(a_n)} 
=
\eta'\left(e_j \wedge e_i \otimes X_1^{(a_1)}\cdots X_j^{(a_j+1)} \cdots x_i^{(0)} \cdots X_n^{(a_n)} \right)
\]
as desired. 
\end{proof}

Under the assumption that $J_q$ is $(x_1^{q}, \dots, x_n^{q})$-compressed and $q \le \frac{s}{2}$, 
Lemma~\ref{TFAErelativelycompressed} applies to the ideal $J_q$ and tells us that the minimal homogeneous generators of $J_q$ in degrees $i\leq \frac{s}{2}$ are simply $x_1^{q}, \dots, x_n^{q}$. 
The next proposition assures us that any additional generators in degrees $i>\frac{s}{2}$ lie only in certain degrees. In view of Remark~\ref{rmk-general}, the result below holds for general choices of $f$ of degree $d$.

\begin{cor}
\label{generatingdegreesofJ}
Let $P=k[x_1, \dots, x_n]$ with $n\geq 3$, assume the notation above, and assume that $f$ is \linkcomp{$q$} (that is, $J_q$ is $(x_1^{q}, \dots, x_n^{q})$-compressed) for some $q$, with  $q \geq (d +n)/(n-2)$ (so that $q \le \frac{s}{2}$). 

Then the minimal homogeneous generators for $J_q$ in degrees $i>\frac{s}{2}$ lie 
\begin{itemize}
\item
in degree $\frac{s}{2}+1$ when $s$ is even and 
\item 
in degrees $\lceil \frac{s}{2}\rceil$ and $\lceil \frac{s}{2}\rceil+1$ when $s$ is odd. 
\end{itemize}
\end{cor}

\begin{proof}
In view of Proposition~\ref{cokernelsvanish}, it follows from Theorem~\ref{resn-general} for $m=\lceil\frac{s}{2}\rceil$, as explained at length in its unpacking in Remark~\ref{rmk-resn-minimal} and summarized in Corollary~\ref{cor-degrees}, that the generators of $(J_q)_{\geq m}$ lie in degrees $m$ and $m+1$.
But by definition of $(x_1^q,\ldots,x_n^q)$-compressed, the only minimal generators of $J_q$ in degrees less than or equal to $m$ are $x_1^{q}, \dots, x_n^{q}$. 
So the minimal generators of $J_q$ lie in degrees $q$, $m$, and $m+1$. 

If, furthermore, $s$ is even then since $\lceil\frac{s}{2}\rceil = \frac{s}{2}$ and $J$ is $(x_1^q,\ldots,x_n^q)$-compressed there are in fact no generators of $J_q$ in degree $m$. 
\end{proof}

The next question is how many such additional generators are there beyond the obvious generators $x_1^{q}, \dots, x_n^{q}$? This is answered by the following result. In view of Remark~\ref{rmk-general}, the result below holds for general choices of $f$ of degree $d$.

\begin{prp}
\label{numberofgeneratorsofJ}
Let $P=k[x_1, \dots, x_n]$ with $n\geq 3$, assume the notation above and assume that $f$ is \linkcomp{$q$} (that is, $J_q$ is $(x_1^{q}, \dots, x_n^{q})$-compressed) for some $q$, with  $q \geq (n+d)/(n-2)$ (so that $q \le \frac{s}{2}$).
Then $x_1^{q}, \dots, x_n^{q}$ are part of a minimal set of generators and the number of remaining generators is described below. 

\begin{itemize}
\item 
If $n=3$ and $s$ is even, then $J_q$ has exactly $2d$ additional minimal generators in degree $\frac{s}{2}+1$ and no others of higher degree. 

\item 
If $n=3$ and $s$ is odd, then $J_q$ has exactly $d$ additional minimal generators in degree $\lceil \frac{s}{2} \rceil$. It also has at most $3d$ additional generators in degree $\lceil \frac{s}{2}\rceil+1$ and no others of higher degree. 

\item
If $n \geq 4$, its dimension is equal to a polynomial in $q$ with leading term
\[
\frac{2}{(n-2)!} \left[ \left(\frac{n}{2}\right)^{n-2}\!\!\!\!-n\left(\frac{n-2}{2}\right)^{n-2} \right] q^{n-2}
\qquad 
{\textrm{ for $s$ even, }}
\]
and bounded by a polynomial in $q$ of degree $n-2$ for $s$ odd. 
Furthermore, the coefficients of these polynomials depend only on $n$ and $d$. 
\end{itemize}
\end{prp}

\begin{proof}
Recall that by Lemma~\ref{phi},  $s=n(q-1)-d$. By Corollary~\ref{generatingdegreesofJ} we know in which degrees the generators lie; it remains to determine the number in each degree. 

When $s$ is even, we check the number of generators of $J_q$ in degree $\frac{s}{2}+1$, and demonstrate that it does not depend on $q$. When $s$ is odd, we check the number of generators of $J_q$ in degree $\lceil \frac{s}{2} \rceil$, and demonstrate that it does not depend on $q$. In this case, there are additional generators in degree $\lceil \frac{s}{2} \rceil+1$, and the number of generators may depend on $q$.

First we examine the case where $n=3$ and $s$ is even. 
Set $a=\frac{s}{2}$. 
We are expecting new generators of $J_q$ in degree $a+1$. 
We find the dimension of $J_q$ in degree $a+1$, which is just the rank of the kernel of the map $\Phi_{a+1}\colon S_{a+1} \to D_{a-1}$. We count ranks along the exact sequence 
\[
0 \to \ker \Phi_{a+1} \to S_{a+1}\xra{\Phi_{a+1}} D_{a-1} \to \coker  \Phi_{a+1} \to 0 
\]
The ranks of $S_{a+1}$ and $D_{a-1}$ are $\binom{a+1+2}{2}$ and $\binom{a-1+2}{2}$, respectively. 
The cokernel of $\Phi_{a+1}$ is the same as the kernel of the dual map $\Phi_{a-1} \colon S_{a-1} \to D_{a+1}$, which has basis given the monomial generators of $(x_1^q,x_2^q,x_3^q)(x_1,x_2,x_3)^{a-1-q}$. So it has $3\binom{a-1-q+2}{2}$ generators. There are also $3\binom{a+1-q+2}{2}$ generators of $J_q$ in degree $a+1$ that are multiples of the old generators $x_1^q,x_2^q,x_3^q$ and do not interest us. 
Subtracting we are left with  
\begin{equation*}
\binom{a+1+2}{2}-\binom{a-1+2}{2}+3\binom{a-1-q+2}{2}-3\binom{a+1-q+2}{2}
\end{equation*}
generators. This expression simplifies to $2d$ generators, independent of $q$.

Next we examine the case where $n=3$ and $s$ is odd. Set $a=\frac{s-1}{2}=\floor{\frac{s}{2}}$. First we consider the new generators of $J_q$ in degree $a+1$. We have a rank of $\binom{a+1+2}{2}$  for $S_{a+1}$, and a rank of $\binom{a+2}{2}$ for $D_{a}$. The cokernel has rank equal to the kernel of the dual map $S_a \to D_{a+1}$ (note that $s-(a+1)=a$), which has rank $3\binom{a-q+2}{2}$. Lastly, there are $3\binom{a+1-q+2}{2}$ generators coming from $(x_1^q,x_2^q,x_3^q)$, that we discount. We are left with
\[\binom{a+1+1}{2}-\binom{a+2}{2}+3\binom{a-q+2}{2}-3\binom{a+1-q+2}{2}\]
generators in degree $a+1$, which simplifies to $d$ generators, independent of $q$.

Next we consider the new generators of $J_q$ in degree $a+2$. By a similar process, we get 
\[\binom{a+2+2}{2}-\binom{a-1+2}{2}+3\binom{a-1-q+2}{2}-3\binom{a+2-q+2}{2}\]
generators, minus whatever new generators came from the degree $a+1$ piece of $J_q$. 
If we call this number of new generators $g_q$, this gives us $3d-g_q$ generators. 

Lastly, we consider the case that $n\geq 4$. 
We give the details when $s$ is even; when $s$ is odd the proof follows from a modification analogous to the one above for $n=3$. 
Set $a=\frac{s}{2}$. 
We are expecting new generators of $J_q$ in degree $a+1$. 
As in the 3 variable case, we use the same exact sequence to compute the dimension of $J_q$ in degree $a+1$, then subtract the portion from the ideal $(x_1^q,\dots,x_n^q)$, finally obtaining \begin{equation}
\label{bigeqn}
\binom{a+1+n-1}{n-1}-\binom{a-1+n-1}{n-1}
+n\binom{a-1-q+n-1}{n-1}-n\binom{a+1-q+n-1}{n-1}
\end{equation}
The first 2 terms can be simplified to 
\begin{align*}
\frac{1}{n-2}\binom{a+n-2}{n-3}(2a+n)
&=
\frac{1}{n-2}
\left(
\frac{1}{(n-3)!} a^{n-3} + 
{\textrm{lower order terms}}
\right)
(2a+n) \\
&=
\frac{2}{(n-2)!} a^{n-2} + 
{\textrm{lower order terms}}
\end{align*}
where we note that lower order in $a$ is equivalent to lower order in $q$. 
Similarly, after we  factor out $-n$ from the last 2 terms, the remaining quantities can be simplified to the same thing with $a-q$ substituted for $a$. 
As $a
=\frac{s}{2}
=\frac{n(q-1)-d}{2}$, the expression (\ref{bigeqn}) simplifies to give the desired asymptotic result. 
\end{proof}

We remark that, for $s$ odd, we are not sure whether $g_q$ depends on $q$. 

\begin{rmk}
For polynomial rings with four or more variables, in contrast to the three variable case, the {\it dimension} of the socle of $R/\fm^{[q]}$ depends on $q$, even when $s$ is even. 
One can see this almost with any random choice of $f$ and $q\leq \frac{s}{2}$, that is, $q \geq (d +n)/(n-2)$.
\end{rmk}

\begin{rmk}
\label{rk:q=d+1}
We make a few side remarks for  polynomial rings in three variables, writing $P=k[x,y,z]$. 
Note that when $q < (n+d)/(n-2)$, we have $q > \frac{s}{2}$, so if $J_q$ is $(x^q,y^q,z^q)$-compressed, $J_q$ has no generators in degrees $ \frac{s}{2}$ or less and is in fact compressed. Resolutions of compressed ideals are described in \cite{ElkhKus-14, MiRa-18}, so we do not consider this case here (outside of this remark).

Nevertheless, when $n=3$ and $s$ is even, we can apply our methods for the case $q=d+1$, as then we have $q=\frac{s}{2}+1$. In this case, $J_q$ has $3q-d$ generators in degree $\frac{s}{2}+1$, which we find by the same method as the proof of Proposition \ref{numberofgeneratorsofJ}, computing  the number of generators as $\rank S_{\frac{s}{2}+1}-\rank D_{\frac{s}{2}-1}$. Three of these generators will be $x^q,y^q,$ and $z^q$, so we can apply Remark \ref{muofJ}. Hence the results of this paper should hold in this case.

However, when $q < d+1$, one has $q > \frac{s}{2}+1$, but by the argument above the minimal generators of $J_q$ still lie in degree $\frac{s}{2}+1$. So the elements $x^q, y^q,$ and $z^q$ are not minimal generators of $J_q$. This means we cannot apply Remark \ref{muofJ}, so we do not get our main results in this case. However, this is again the  compressed case which has already been studied.
\end{rmk}

Lastly, we focus for a moment on the case of 3 variables, writing 
\[
P=k[x,y,z].
\]
In \cite[1.1]{KU-09} Kustin and Ulrich show that under some mild conditions two Frobenius powers of the maximal ideal have the same  graded Betti numbers up to a particular given shift if and only if their quotient rings have isomorphic socles up to a particular shift. 
Here, just for the sake of curiosity,
we directly verify that their desired isomorphism of socles holds.

However, in Section~\ref{sec:frobenius} we prove {\it directly} that, whenever $f$ is \linkcomp{$q$} for various $q$, the conclusion for the Betti numbers follows from our results in this section (also without the mild hypotheses that their result would require). In view of Remark~\ref{rmk-general}, the result below holds for general choices of $f$ of degree $d$, and for very general choices of $f$ of degree $d$, it holds for all $q_0$.

\begin{cor}
\label{KUsoc}
Assume the notation above and fix a power $q_0\ge \deg f+3$ of $p$.  
Assume that $f$ is both \linkcomp{$q_0$}  and \linkcomp{$q_0q$} for some $q>1$ 
(that is, $J_{q_0}$ is $(x^{q_0}, y^{q_0}, z^{q_0})$-compressed and $J_{q_0q}$ is $(x^{q_0q}, y^{q_0q}, z^{q_0q})$-compressed). 

If $s$ is even, or if $s$ is odd and the minimal number of generators of $J_{q_0}$ and $J_{q_0q}$ 
in degree $\lceil \frac{s}{2} \rceil +1$ are equal, then 
\[
\soc R/\fm^{[q_0q]} 
\cong 
\soc R/\fm^{[q_0]}\left(-\frac{3}{2}q_0(q-1)\right)
\] 
  as graded vector spaces. 
\end{cor}
\begin{proof}
Set $q=p^e$ and $q_0=p^{e_0}$. 
We calculate $\soc(R/\fm^{[\tilde{q}]})$ for any $\tilde{q}=p^{\tilde{e}}\geq q_0$ and then specify to the values $\tilde{q}=q_0$ and $\tilde{q}=q_0q=p^{e_0+e}$ to obtain the desired result. 
To do this, recall that $R/\bar{I}^{[q]}\cong P/I_q$ and apply Remark~\ref{muofJ} 
with $I=I_{\tilde{q}}$, $J=J_{\tilde{q}}$, and $({\bf x}^{\tilde q})=(x_1^{\tilde{q}}, \dots, x_n^{\tilde{q}})$ to obtain the degrees of the generators of $\soc(R/\fm^{[\tilde{q}]})$ from the degrees generators of the link $J_{\tilde{q}}$. Note that the socle degree of the complete intersection is $P/({\bf x}^q)$ is $3(\tilde{q}-1)$.  
By Proposition \ref{numberofgeneratorsofJ} (plus the additional hypothesis when $s$ is odd), the number of generators of $J_{\tilde{q}}$ in each degree does not depend on $\tilde{q}$. 
Hence the dimension of
$\soc (R/\fm^{[\tilde{q}]})$ does not either. 

The degrees of the generators of $J_{\tilde{q}}$ do however depend on $\tilde{q}$. 
First suppose that $s$ is even.
The same proposition yields that the socle generators have degrees equal~to 
\begin{equation}\label{sigma-even}
3(\tilde{q}-1) -\frac{1}{2}\lceil 3(\tilde{q}-1)+d \rceil -1
=
\frac{1}{2} (3(\tilde{q}-1)+d) - 1
\end{equation}
where $d=\deg f$. 
Now suppose that $s$ is odd. Then the socle generators  have degrees
\begin{align}\label{sigma-odd}
\begin{split}
&3(\tilde{q}-1) -\frac{1}{2}\lceil 3(\tilde{q}-1)+d \rceil
=
\frac{1}{2} (3(\tilde{q}-1)+d) - \frac{1}{2} , 
~\text{ \ or}\\
&3(\tilde{q}-1) -\frac{1}{2}\lceil 3(\tilde{q}-1)+d \rceil -1
 =
\frac{1}{2}(3(\tilde{q}-1)+d) -\frac{3}{2}.
\end{split}
\end{align}
Subtracting the expressions obtained by plugging in $\tilde{q}=q_0=p^{e_0}$ and $\tilde{q}=q_0q=p^{e_0+e}$, one gets degree shifts equal to  
\[
\left(\frac{1}{2}(3(p^{e_0+e}-1)+d)-t\right)-\left(\frac{1}{2}(3(p^{e_0}-1)+d)-t\right)
\]
\[
=\frac{3}{2}(p^{e_0+e}-p^{e_0})=\frac{3}{2}p^{e_0}(p^e-1)=\frac{3q_0(q-1)}{2}
\]
where $t =\frac{1}{2}$, $1$ or $\frac{3}{2}$. The last assertion follows from the fact that $P/({\bf x}^{q_0})$ is resolved by the Koszul complex and therefore its back twist is $b=3q_0$.
\end{proof}

\begin{rmk}
\label{char-0}
We remark that there are characteristic zero versions of the results in this section if one uses the bracket powers  $\fm^{[t]}=(x_1^t,\cdots,x_n^t)$ of the maximal ideal for every integer $t$, instead of the Frobenius powers. We are not stating these results as we are  mostly interested in the positive characteristic case but interested readers can prove them using similar arguments with appropriate modifications. 

Note that in characteristic zero, by \cite[Lemma 4.15 and Thm 4.16]{IarKan-99}, for a general choice of $f$, the ideal $J_t$ is $\fm^{[t]}$-compressed, hence the assertions in the  characteristic zero versions of the results hold for general $f$.
\end{rmk}

  \section{Application to Frobenius powers}
  \label{sec:frobenius}

We now specialize to the case of three variables, where our socle results are strongest, applying these to obtain graded Betti numbers of Frobenius powers of the maximal ideal and Hilbert-Kunz functions. 
Let 
\[
R=P/(f) 
\qquad 
{\textrm{with \ }}
P=k[x,y,z]
\]
be a standard graded hypersurface ring over a field $k$ of characteristic $p>0$ with homogeneous maximal ideal $({\bf x})$ such that $f\in({\bf x})^2$. 

In this section, for sufficiently randomly chosen polynomials $f$, we find the format of the free resolutions of the quotients $R/\fm^{[q]}$ by the Frobenius powers of $\fm$ over both rings $P$ and $R$, as well as the structure of each associated matrix factorization (which turns out to be a pair of skew-symmetric matrices, one of which is linear with Pfaffian equal to $f$) and the Hilbert-Kunz function for $\fm$. In fact, we prove these results for all $f$ with the property that $(({\bf x}^q):f)$ is $({\bf x}^q)$-compressed, i.e. $f$ is \qlinkcomp. 
By Remark \ref{rmk-general}, for a fixed even $d$ and odd $p>d$, very general choices of $f$ are \qlinkcomp\ for all $q$.

The organization of the section is as follows: 
In the first part we explore what can be said without the assumption that $f$ is \qlinkcomp, 
specifically about the resolution over $P$ using linkage, the Buchsbaum-Eisenbud structure theorem for Gorenstein ideals of codimension three, and the corresponding dg-algebra structure. 
In the second part, we determine the structure of the resolution over $R$ using the construction of Shamash (Eisenbud-Avramov-Buchweitz). 
In the third part, for $f$ \qlinkcomp, we apply this in concert with Proposition~\ref{numberofgeneratorsofJ} to deduce Theorem~\ref{gradedBetti}, 
giving precise graded Betti numbers and the regularity and hence the stability of Betti numbers of Frobenius powers of the maximal ideal. 
In the fourth part, again for $f$ \qlinkcomp, we use these results to compute the Hilbert-Kunz function of the maximal ideal in Theorem~\ref{HKfunction}. 

We note that in the first two parts we develop resolutions for {\it any} graded complete intersection ideal, not just $(x^q,y^q,z^q)$. Then in the last two parts we specialize to the case of $(x^q,y^q,z^q)$ for our applications. 

\begin{ntn}
Let $d=\deg f$ and let $q$ be a power of $p$.
In this section we work with the following graded ideals of the polynomial ring $P$: the complete intersection ideal 
\[
\mathfrak{c}=(c_1,c_2,c_3)
\]
where $\deg c_i=r_i$
and the linked ideals 
\[
I=\mathfrak{c}+(f)
\qquad {\textrm{and}} \qquad 
J=(\mathfrak{c}:I)
\] 
Since $I$ is an almost complete intersection ideal, $P/J$ is a Gorenstein ring. 

Later (in Parts 3 and 4 of this section) we will specialize to the case
\[
\mathfrak{c}=({\bf x}^q)=(x^q,y^q,z^q)
\] 
and the linked ideals 
\[
I_q=({\bf x}^q)+(f)=(x^q,y^q,z^q,f) 
\qquad {\textrm{and}} \qquad 
J_q=(({\bf x}^q):I_q)
\] 
so that we can apply the results from Section~\ref{sec:socle}. 

Assume $c_1,c_2,c_3$ are part of a minimal generating set for $J$.
(This holds automatically in the special case where $J=J_q$ is $(x^q,y^q,z^q)$-compressed and $q \geq d+3$. See Remark \ref{rem-qvsd} for details.)
Fix a set of minimal generators for $J$ as follows 
\[
J=(c_1,c_2,c_3,w_1,\dots,w_m)
\qquad {\textrm{so that}}
\qquad 
m=\mu(J)-3.
\] 
 
\end{ntn}

\vspace{3mm}
\centerline{\bf Part 1: Resolutions over $P$}
\vspace{3mm}

In this section we construct the $P$-free resolution of $P/I$ in terms of that of $P/J$ via the theory of linkage. 
We follow the ideas and techniques from Avramov \cite{Av81}  using the dg algebra structure defined by Buchsbaum and Eisenbud in \cite[\S 4]{BE-77} to derive explicit formulas for our particular situation; see also \cite[Section 8.4]{Av78}. Note also that the results are similar, but derived in the opposite order, to those in \cite[Lemma 2.3]{KRV-12}. 

First recall that since $P/J$ is Gorenstein of codimension three and its resolution is given by the well-known Buchsbaum-Eisenbud structure theorem. 
For the reader's convenience, we now modify the proof given in \cite[Theorem 3.4.1b]{bruns-herzog} to make it valid in the graded setting and for the order of generators for $J$ that we prefer.  

\begin{lem}[Buchsbaum-Eisenbud \cite{BE-77}, graded version]
\label{lem-brunsherzog}
\ \\ 
There is a graded $P$-free resolution of $P/J$ of the form
\[
0 \to P \xrightarrow{\partial_3} P^3\oplus P^m \xrightarrow{\partial_2} P^3\oplus P^m \xrightarrow{\partial_1} P \to 0,
\]
where 
\[
\partial_1=[c_1,c_2,c_3,w_1,w_2,\cdots,w_m], \qquad \partial_3=\partial_1^T, 
\]
and the map $\partial_2$ is given by an skew-symmetric matrix, which we can describe as
\[
X=
\begin{bmatrix} 
* & -\psi^T \\ 
\psi & \phi 
\end{bmatrix},
\]
where $*$ is a $3 \times 3$ block, $\psi$ is $m \times 3$, and $\phi$ is $m \times m$, and furthermore, the $j$th entry of $\partial_1$ equals $(-1)^{j+1}\Pf_j(X)$ for each $j=1,\dots,m+3$. 
\end{lem}

\begin{proof}
Take a graded minimal free resolution $F$
\[0 \to F_3 \to F_2 \to F_1 \to F_0 \to 0\] 
of $P/J$, fixing bases $e_1,\ldots,e_{m+3}$ and $c$ of $F_1$ and $F_0=R$, respectively, so that the $e_i$ map to the generators $c_1,c_2,c_3,w_1,\ldots,w_m$ of $J$.  By  \cite[3.4.3]{bruns-herzog}, we can assume there is a (possibly non-associative) graded product on $F$. In particular, we get pairings 
\[F_i \otimes F_{3-i} \to F_3 \cong P,\]
and by tensor-hom adjointness this gives us graded homomorphisms 
\[
s_i: F_i \to \Hom_P(F_{3-i},F_3) \cong F_{3-i}^*.
\]
By  \cite[3.4.4]{bruns-herzog}, we get an isomorphism of complexes $t:F \to F^*$ where $t_i=\pm s_i$. 
\\
Consequently, the elements  $s_1(e_1),\ldots,s_1(e_m)$ form a basis for $F_2^*$. Let $f_1,\ldots,f_m$ be a basis for $F_2$ such that $f_i^*=s_1(e_i)$. By construction and by the properties given in \cite{bruns-herzog}, we have $f_ie_j=\delta_{ij}g$, where $g$ is the basis element for $F_3$ such that $s_3(g) \in F_0^*$ is the dual of $c \in F_0$. The rest of the proof follows as in \cite[Page 124]{bruns-herzog}. 
\end{proof}

\begin{dfn} 
\label{def-Pfaffian-adjoint}
Let $X$ be a skew-symmetric matrix. We denote by $Pf_{i_1\cdots i_j}(X)$ the Pfaffian of the matrix obtained from $X$ by removing the rows and columns indexed by $i_1,\dots, i_j$.

The Pfaffian adjoint of $X$, denoted, $X^\vee$, is the skew-symmetric matrix with entries
\[X^\vee_{i,j}:= \begin{cases}
    (-1)^{i+j}Pf_{i j}(X) & i<j \\
    0 & i=j \\
    (-1)^{i+j+1}Pf_{i j}(X) & i>j. \\
    \end{cases}
    \]
\end{dfn}
 
\begin{dfn} 
\label{def-BE-dga}
Buchsbaum and Eisenbud define a dg-algebra structure on their resolutions  \cite[\S 4]{BE-77}. The formulas that we recall here are taken from \cite[Section 8.4]{Av78}): 
Let $\{C\}$, $\{B_i\}_{i=1}^{m+3}$ and $\{A_i\}_{i=1}^{m+3}$ be the standard bases for the free modules, respectively from left to right, in the resolution in Lemma~\ref{lem-brunsherzog}. For $1 \leq i,j \leq m+3$, define  the products
\begin{align*}
    A_iA_j&=\sum_{p=1}^{2d+3} \sigma_{ijp}\Pf_{ijp}(X)B_p\\
    B_hA_i&=A_iB_h-\delta_{ih}C
\end{align*}
where $\sigma_{ijp}$ is the sign of the permutation sending the numbers $1,2,\ldots,m+3$ to the numbers $i,j,p$ followed by the remaining numbers in increasing order if $i,j,$ and $p$ are distinct, and 0 otherwise. These induce a dg-algebra structure on the minimal free resolution of $P/J$. 
\end{dfn}
 
In order to find a resolution of $P/I$ via linkage theory, as developed algebraically in \cite{PesSzp-74}, we need an explicit map from the Koszul complex on the complete intersection ideal $\mathfrak{c}$ to the resolution of $P/J$ lifting the natural surjection $P/\mathfrak{c}\onto P/J $. 
We will use the dg-algebra structure defined above to achieve this. 

\begin{prp}
\label{prp-complexmap}
Let $F$ be the free resolution of $P/J$ from Lemma \ref{lem-brunsherzog} with $\psi$ and $\phi$ as defined there. 
Let $K$ be the Koszul complex  $K(c_1,c_2,c_3;P)$. 
Then there is a lifting of the identity map $K_0 \xra{=} F_0$ to a map of complexes $\alpha\colon K \to F$ such that in the standard basis of the Koszul complex and the basis for $F$ given in Lemma~\ref{lem-brunsherzog}:
\begin{enumerate}
\item $\alpha_1$ is given by the matrix \[
\begin{bmatrix}
\iden_3 \\
0 
\end{bmatrix}
\]
where $\iden_3$ is the $3\times3$ identity matrix,
\item $\alpha_2$ is given by the product of matrices 
\[
\begin{bmatrix}
\Pf(\phi)\iden_3 \\
-\phi^\vee\psi 
\end{bmatrix}
[\widetilde{\iden}_3]
\]
where $\widetilde{\iden}_3$ is the $3\times3$ identity matrix with middle entry changed to $-1$, and 
\item $\alpha_3$ is multiplication by $\Pf(\phi)$. 
\end{enumerate}
\end{prp}

\begin{proof}
Consider the Koszul complex $K=K(c_1,c_2,c_3;P)$ 
\[0 \lra P \xra{d_3} P^3 \xra{d_2} P^3 \xra{d_1} P \to 0\]
resolving $P/\mathfrak{c}$ over $P$. 
By using the dg-algebra structures of the two resolutions $K$ and $F$, we will construct a map of complexes $\alpha\colon K \to F$ given by 
\[
 \xymatrixrowsep{1.5pc} \xymatrix{
0\ar[r]&P\ar[r]^{d_3}\ar[d]^{\alpha_3}&P^3\ar[r]^{d_2}\ar[d]^{\alpha_2}&P^3\ar[r]^{d_1}\ar[d]^{\alpha_1}&P\ar[r]\ar[d]_{=}^{\alpha_0}&0
\\
0 \ar[r]&P\ar[r]^{\partial_3}&P^3\oplus P^m\ar[r]^{\partial_2}&P^3\oplus P^m\ar[r]^{\partial_1}&P\ar[r]&0}
\]
that lifts the identity map in homological degree 0. 

First, let $T_1,T_2,T_3$ be the basis elements of $K_1$ that are sent to $c_1,c_2,c_3$ in $P$. 
To make the right hand square commute, set $\alpha(T_i)=A_i$ for $i=1,2,3$. This determines the map $\alpha_1$  whose matrix is
\[\begin{bmatrix}
\iden_3 \\
0 \\
\end{bmatrix}\]
where $\iden_3$ is the $3 \times 3$ identity matrix.

Then, as the Koszul complex is a free graded  dg algebra on $T_1, T_2, T_3$ and $F$ is a dg algebra with structure as described in Definition~\ref{def-BE-dga}, this map extends uniquely to a map of dg algebras $\alpha\colon K \to F$.
We now compute it explicitly on the bases $\{T_1T_2,T_1T_3,T_2T_3\}$ in degree two and $T_1T_2T_3$ in degree three.  

First we find $\alpha_3$.  Since $\alpha$ is a dg-algebra map, we have 
\[
\alpha_3(T_1T_2T_3)=\alpha_1(T_1)\alpha_1(T_2)\alpha_1(T_3)=A_1A_2A_3.
\]
Using the dg algebra structure on $F$ described in Definition~\ref{def-BE-dga}, we find that 
\begin{equation}
\begin{aligned}
\alpha_3(T_1T_2T_3)&=A_1A_2A_3=
\left(\sum_{p=1}^{m+3}\sigma_{12p}
\Pf_{12p}(X)B_p\right) A_3 \\
&=\sum_{p=1}^{m+3}
\sigma_{12p}\Pf_{12p}(X)\delta_{p3}C
=\sigma_{123}\Pf_{123}(X)C \\
&=\Pf_{123}(X)C=\Pf(\phi)C.
\end{aligned}
\end{equation}

Similarly, to find $\alpha_2$ we compute 
\[\alpha_2(T_1T_2)=\alpha_1(T_1)\alpha_1(T_2)=A_1A_2=\sum_{p=1}^{m+3} \sigma_{12p}\Pf_{12p}(X)B_p.\]
To find this sum we first compute the  the first nonzero term: 
\[\sigma_{123} \Pf_{123}(X)B_p= \Pf_{123}(X)B_p= \Pf(\phi)B_p,\]
For the remaining terms we compute $\Pf_{12p}$ for $p\geq 4$. Expanding along the third row and column, we get 
\[
\Pf_{12p}(X) = \sum^ {m+3}_{\substack{  \ell =4  \\ \ell \ne p}} (-1)^{\tilde{\ell}} X_{3\ell}\Pf_{123\ell p}(X),\]
where $\tilde{\ell}=\ell$ if $\ell <p$ and $\tilde{\ell}=\ell-1$ if $\ell > p$ and  $X_{3\ell}$ is the $(3,\ell)$ entry of $X$. Thus 

\[\Pf_{12p}(X) = \sum^ {m+3}_{\substack{  \ell =4  \\ \ell \ne p}} (-1)^{\tilde{\ell}} \Pf_{\ell p}(\phi),\]
 This simplifies to
\[\sum^ {m+3}_{\substack{  \ell =4}} (-1)^{\tilde{\ell}}X_{3\ell}(-1)^{p+\ell}(\phi^\vee)_{\ell p}.\]
where $\phi^\vee$ is the adjoint of $\phi$ and, by definition, 
\begin{equation*}
(\phi^\vee)_{\ell p} =
   \begin{cases}
     \Pf_{\ell p}(\phi) & \textrm{if}~\ell < p\\
     0 & \textrm{if}~\ell=p\\
     -\Pf_{\ell p}(\phi)& \textrm{if}~\ell > p.
    \end{cases} 
\end{equation*}
Hence one has
\[ \Pf_{12p}(X) =\sum^ {m+3}_{\substack{  \ell =4 }} (-1)^p(X_{3\ell})(\phi^\vee)_{\ell p}.\]
Notice that $(X_{3\ell})_{\ell=4}^{m+3}$ is the 3rd row of $-\psi^T$, whereas $\left((\phi^\vee)_{\ell p}\right)_{\ell=4}^{m+3}$ is the $p$th column of $\phi^\vee$. Hence this simplifies to
\[ \Pf_{12p}(X)=(-1)^p(-\psi^T\phi^\vee)_{3p}.\]

Hence 
\[
\begin{aligned}
\alpha_2(T_1T_2) &=\Pf(\phi)B_3+\sum_{p=4}^{m+3} \sigma_{12p}(-1)^p(-\psi^T\phi^\vee)_{3p}B_p \\
&=\Pf(\phi)B_3+\sum_{p=4}^{m+3} (\psi^T\phi^\vee)_{3p}B_p \\
&=\Pf(\phi)B_3+\sum_{p=4}^{m+3} (-\phi^\vee \psi)_{p3}B_p,
\end{aligned}
\]
where the last equality follows because $(\psi^T\phi^\vee)^T=(\phi^\vee)^T\psi=-\phi^\vee\psi$. This gives the third column of the matrix of $\alpha_2$.

The computations of $\alpha_2(T_1T_3)$ and $\alpha_2(T_2T_3)$ are similar, except that for the former one gets an extra minus sign.
Hence the matrix of the map $\alpha_2:P^3 \to P^3\oplus P^m$ is
\[
\begin{bmatrix}
\Pf(\phi)\iden_3 \\
-\phi^\vee \psi 
\end{bmatrix}
[\widetilde{\iden}_3]
\]
where $\widetilde{\iden}_3$ is the $3\times3$ identity matrix with middle entry changed to $-1$.  
\end{proof}

Now we use the explicit map from Proposition~\ref{prp-complexmap} to obtain the resolution of $P/I$ from that of the linked ring $P/J$. 
Recall that $\phi, \psi$ are the blocks appearing in the middle matrix of the resolution of $P/J$ as defined in Lemma~\ref{lem-brunsherzog}.

\begin{prp}
\label{I-Presn}
Assume that $r_i > d$ for $i=1,2,3$.
The minimal homogeneous resolution $G$ of $P/I$ over $P$ is of the form
\[
0 \lla P 
\xla{
\begin{bmatrix}
{\bf c} & \!\! vf 
\end{bmatrix}
} 
\begin{matrix}
P^3 \\\oplus\\ P
\end{matrix}
\xla{
\begin{bmatrix}
\psi^T\phi^\vee & vf\iden_3 \\
-{\bf w} & -{\bf x}^q
\end{bmatrix}
} 
\begin{matrix}
P^m \\\oplus\\ P^3
\end{matrix}
\xla{
\begin{bmatrix}
\phi \\ -\psi^T 
\end{bmatrix}
} 
P^m \lla 0
\]
with $\Pf(\phi)=vf$ for some unit $v\in k$, $[{\bf c}]=[c_1 \ c_2 \ c_3]$ and $[{\bf w}]=[w_1~\cdots~w_m]$. 
\end{prp}

\begin{proof}
As  the ideals $I$ and $J$ are linked via the complete intersection ideal $\mathfrak{c}$, 
we use the method developed in \cite{PesSzp-74} to go from a resolution of $P/J$ to one of $P/I$. 

Starting with the map of complexes from Proposition \ref{prp-complexmap} 
we take the mapping cone and dualize it to get a resolution $G$ of $P/I$ over $P$ as the cone of
\[
 \xymatrixrowsep{1.5pc} \xymatrix{
0
&P^*\ar[l]
&(P^3)^*\ar[l]_{d_3^*}
&(P^3)^*\ar[l]_{d_2^*}
&P^*\ar[l]_{d_1^*}
&0\ar[l]
\\
0 
&P^*\ar[l]\ar[u]^{\alpha_3^*}
&(P^3)^*\oplus (P^m)^*\ar[l]_{\partial_3^*}\ar[u]^{\alpha_2^*}
&(P^3)^*\oplus (P^m)^*\ar[l]_{\partial_2^*}\ar[u]^{\alpha_1^*}
&P^*\ar[l]_{\qquad \partial_1^*}\ar[u]^{\alpha_0^*}_{=}
&0\ar[l]
}
\]
Now the non-minimal pieces of the resolution can be removed as follows:
$\alpha_0^*$ is the identity map, so those copies of $P$ are removed. The map $\alpha_1^*$ identifies the two copies of $P^3$ so those can be removed as well. 
From the formulas in Proposition~\ref{prp-complexmap}, we see that the maps $\alpha_2$ and $\alpha_3$ are minimal.
This yields a minimal resolution as the cone of the following diagram, since the matrix of $\partial_2^*$ equals $X^T=-X=
{\small 
\begin{bmatrix} -* & \psi^T \\ -\psi & -\phi \\
\end{bmatrix}
}
$. 

\[
 \xymatrixrowsep{1.5pc} \xymatrixcolsep{3.5pc}\xymatrix{
0
&P^*\ar[l]
&(P^3)^*\ar[l]_{d_3^*}
&
&
\\
0 
&P^*\ar[l]\ar[u]^{\alpha_3^*}
&(P^3)^*\oplus (P^m)^*\ar[l]_{
\tiny
\begin{bmatrix}
{\bf c} & {\bf w}
\end{bmatrix}
}
\ar[u]^{\alpha_2^*}
&(P^m)^*\ar[l]_{
\tiny 
\begin{bmatrix}
\psi^T \\ -\phi 
\end{bmatrix}
}
&0\ar[l]
}
\]

In particular, the sum of the images of $\alpha_3^*$ and $d_3^*$ must equal $I=(f)+\mathfrak{c}$. 
By Proposition~\ref{prp-complexmap}, the map  $\alpha_3^*$ is multiplication by $\Pf(\phi)$ and $d_3^*$ is given by the matrix $[c_1,-c_2,c_3]$ in the dual of the usual basis for the Koszul complex. 
Hence $\Pf(\phi)$ and $f$ must generate the same ideal modulo $\mathfrak{c}$, so we must have $\Pf(\phi)=vf+h_1c_1+h_2c_2+h_3c_3$ for some unit $v \in P$. But by Lemma \ref{lem-brunsherzog}, $\phi$ must be homogeneous, and so $\Pf(\phi)$ is homogeneous. Since $f$ is homogeneous and $d<r_i$ for $i=1,2,3$ by assumption, we must have $h_i=0$ for all $i$. Hence $\Pf(\phi)=vf$, as desired. 

Last, if one changes basis on the $P^3$ in the top row by the matrix $\widetilde{\iden}_3$,  
the map $d_3^*$ has matrix  $[c_1 \ c_2 \ c_3]$ and the map $\alpha_2^*$ has matrix 
\[
\begin{bmatrix}
\Pf(\phi)\iden_3 \\
-\phi^\vee\psi 
\end{bmatrix}^T
=
\begin{bmatrix}
vf\iden_3 & \psi^T\phi^\vee
\end{bmatrix}
\] 
since $(-\phi^\vee\psi)^T
=-\psi^T(\phi^\vee)^T
=\psi^T\phi^\vee$ as $\phi$ is skew-symmetric. 

This agrees with the resolution in the statement once one notes that the summands in $P^3\oplus P^m$ are now listed in the opposite order and one puts minus signs on the second row when forming the totalized complex. 
\end{proof}

We rewrite the resolution of $P/I$ from Proposition~\ref{I-Presn} as follows for use in the next part  
\begin{equation*}
  \xymatrixrowsep{.3pc} 
  \xymatrixcolsep{2.0pc}
\xymatrix{ 
    &&P^3\ar[dll]_{[{\bf c}]}&&P^m\ar[ll]_{\psi^T\phi^{\vee}}
    \ar[ddll]_{\qquad -[{\bf w}]}&&
    \\
    P&&\oplus&&\oplus&&P^m
    \ar[ull]_{\phi}\ar[dll]^{-\psi}
    \\
     &&P\ar[ull]^{[vf]}&&P^3\ar[ll]^{-[{\bf c}]}
     \ar[uull]^{ \qquad
     vf\iden_3}&&
}
\end{equation*}
and rename the summands as follows to distinguish between the two free modules $P^3$ and likewise for $P^m$.
\begin{equation}
\label{dgm-I-Presn}
  \xymatrixrowsep{.3pc} 
  \xymatrixcolsep{2.0pc}
\xymatrix{ 
    &&G_1'\ar[dll]_{[{\bf c}]}&&G_2'\ar[ll]_{\psi^T\phi^{\vee}}
    \ar[ddll]_{\qquad -[{\bf w}]}&&
    \\
    G_0&&\oplus&&\oplus&&G_3
    \ar[ull]_{\phi}\ar[dll]^{-\psi}
    \\
     &&G_1''\ar[ull]^{[vf]}&&G_2''\ar[ll]^{-[{\bf c}]}
     \ar[uull]^{\qquad vf\iden_3}&&
}
\end{equation}

\vspace{3mm}
\centerline{\bf Part 2: Resolutions over $R$}
\vspace{3mm}

In this section we follow the classic construction of Shamash \cite[\S 3]{JSh69} to obtain a minimal resolution of $P/I = R/\mathfrak{c}$ over $R=P/(f)$ from the resolution over $P$ given in Part 1. 
For more details, see also Remark 2.2.1 and Theorem 9.1.1 in Avramov's exposition  \cite{Av10}. 

We begin with the resolution $G$ of $P/I$ from Proposition~\ref{I-Presn} as rewritten in Diagram~\ref{dgm-I-Presn}. 

The first step is to find a square-zero null homotopy for multiplication by $f$ on $G$. 

\begin{lem}
\label{homotopies}
With the notation above, the following assignments define a null homotopy $\sigma$ for multiplication by $f$ on $G$.
\[
\sigma_0 =
\begin{bmatrix}
0 \\ 1
\end{bmatrix}
\qquad 
\sigma_1 =
\begin{bmatrix}
0 & 0 \\
\iden_3 & 0 
\end{bmatrix}
\qquad 
\sigma_2 =
\begin{bmatrix}
\phi^\vee \\ 0
\end{bmatrix}
\]
In other words, define 
$\sigma_0=1$ on $G_0 \to G_1''$
(the identity map), 
$\sigma_1=\iden_3$ on $G_1' \to G_2''$ 
(the identity map), 
$\sigma_2=\phi^\vee$ on $G_2' \to G_3$ 
and all remaining maps between summands equal to 0.

Furthermore, $\sigma^2=0$. 
\end{lem}

\begin{proof}
Using the formulas for the differential $\partial$ of $G$ given in Proposition~\ref{I-Presn}, one can easily check that $\sigma\partial+\partial\sigma=f$ using the fact that $\phi\phi^\vee = \phi^\vee\phi = f\iden_3$ recalled in Definition~\ref{def-Pfaffian-adjoint}.  
\end{proof}

We are now ready to give the form of the $R$-free resolution of $R/\mathfrak{c}=P/I$. 

\begin{prp}
\label{shamash}
Assume that $r_i > d$ for $i=1,2,3$. The $R$-free resolution of $R/\fc=P/I$ is
\[
\cdots \lra 
R^m \xra{{\phi}} 
R^m \xra{{\phi^\vee}} 
R^m \xra{{\phi}} 
R^m \xra{
{\psi^T\phi^\vee}} 
R^3 \xra{[{\bf c}]} 
R \lra 0 
\]
where by $\phi$, $\phi^\vee$, $\psi^T$, and ${\bf c}$ above we mean the images in $R$ of these $P$-matrices. 

In addition, the homogeneous skew-symmetric matrix $\phi$ has Pfaffian 
equal to $vf$ for some unit $v\in k$, and the pair $(\phi,\phi^\vee)$ of matrices over $P$ is the matrix factorization of $vf$ over $P$ associated to the periodic portion. 
\end{prp}

\begin{proof}
We apply the construction of Shamash \cite{JSh69}. We take as input the resolution $G$ of $P/I$ from Proposition~\ref{I-Presn} and the homotopy $\sigma$ from Lemma~\ref{homotopies}. 
Let $D(R)$ be the divided power $R$-algebra on one variable, call it $T$. 
As $\sigma^2=0$, the construction requires no higher homotopies. 
It yields a resolution of the form 
\[
G \otimes_P D(R),
\]
with differential given by 
\begin{equation}
\label{differential}
\partial(g\otimes T^{(i)}) 
= \partial(g)\otimes T^{(i)} 
+ (-1)^{|g|}\sigma(g)\otimes T^{(i-1)}
{\textrm{ for }} g\in G, i\geq 0.
\end{equation}
One may rewrite this as follows. Set $\overline{G}=G\otimes_P R$ and note that, on the underlying modules, we have 
\[
G \otimes_P D(R) 
\cong \overline{G} \otimes_R D(R)
\cong \bigoplus_{i\geq 0} \left( \overline{G} \otimes_R RT^{(i)} \right),
\] 
each of whose summands is a copy of the complex $\overline{G}$, that is, 
\begin{equation*}
  \xymatrixrowsep{0.3pc} 
  \xymatrixcolsep{2.0pc}
\xymatrix{ 
    &&R^3\ar[dll]_{[{\bf c}]}&&R^m\ar[ll]_{\psi^T\phi^{\vee}}
    \ar[ddll]_{\qquad -[{\bf w}]}&&
    \\
    R&&\oplus&&\oplus&&R^m
    \ar[ull]_{\phi}\ar[dll]^{-\psi}
    \\
     &&R\ar[ull]^{0}&&R^3\ar[ll]^{-[{\bf c}]}
     \ar[uull]^{\qquad\qquad 0}&&
}
\end{equation*}
Equipping the direct sum of these with the Shamash differential (\ref{differential}), we get that the totalization of the following diagram, which consists of shifted copies of the diagram above, is a resolution over $R$ (note that the sign in the formula for the  differential changes the map $\iden_3$ in $\sigma_1$ to $-\iden_3$). 
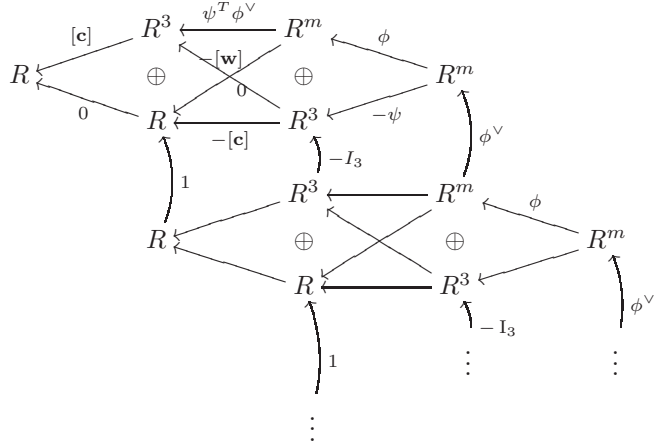
\begin{equation*}
  \xymatrixrowsep{0.3pc} 
  \xymatrixcolsep{3.0pc}
\xymatrix{ 
    &R^3\ar[dl]_{[{\bf c}]}&R^{m}\ar[l]_{\psi^T\phi^{\vee}}\ar[ddl]_{\qquad -[{\bf w}]}&&&&
    \\
    R&\oplus&\oplus&R^{m}\ar[ul]_{\phi}\ar[dl]^{-\psi}& &&
    \\
    &R\ar[ul]^0&R^3\ar[l]^{-[{\bf c}]}\ar[uul]^{\qquad \qquad 0}&&&&
    \\
    &&&&&&
    \\
    &&R^3\ar[dl]\ar@/_/[uu]_{-I_3}&R^{m}\ar[l]\ar[ddl]
    \ar@/_/[uuu]_{\phi^\vee}&&&
    \\
    &R\ar@/_/[uuu]_1&\oplus&\oplus&R^{m}\ar[ul]_{\phi}
    \ar[dl]&& 
    \\
    & &R\ar[ul]&R^3\ar[l]\ar[uul]&&&
    \\
    &&&&&&
    \\
    &&&\ \ \ \vdots\ar@/_/[uu]_{-\iden_3}&\ \ \vdots\ar@/_/[uuu]_{\phi^\vee}&&
    \\
    &&\ \ \vdots\ar@/_/[uuu]_{1}&&&&
    \\
}
\end{equation*}

It is clear that one can reduce this to a minimal resolution by removing the parts identified under the maps $1$ and $-\iden_3$ in the picture. To do this, one rebuilds the complex as the limit of its subcomplexes $\bigoplus_{i\geq 0} \left( \overline{G} \otimes_R RT^{(i)} \right)$, each built as a mapping cone from the previous one and each of which can be reduced to a minimal complex using Proposition~\ref{prp-min}. It is easy to see all the relevant free modules drop out and the inherited differentials are what is listed there, that is, the induced differentials from cokernels to kernels in the statement of Proposition~\ref{prp-min} turn out to be zero here since their domain or codomain is zero.  The result is the desired complex. 
\end{proof}

\vspace{3mm}
\centerline{\bf Part 3: Conclusions for the $({\bf x}^q)$-compressed Case}
\vspace{3mm}

We now specialize to the case of 
\[
\mathfrak{c}=({\bf x}^q)=(x^q,y^q,z^q)
\] 
and the linked ideals 
\[
I_q=({\bf x}^q)+(f)=(x^q,y^q,z^q,f) 
\qquad {\textrm{and}} \qquad 
J_q=(({\bf x}^q):I_q).
\] 
In this part, we apply our technical results on the 
 generators of the link $J_q$ 
proved in Section~\ref{sec:socle} under the assumption that $J_q$ is $({\bf x}^q)$-compressed. 
In view of Remark~\ref{rmk-general}, this condition holds for all $q$ for very general choices of $f$ of degree $d<p$. 
We use these results to derive the  \emph{graded} Betti numbers 
 of $I_q$
and facts about the matrix factorizations of the Frobenius powers of the homogeneous maximal ideal of the hypersurface ring $R=Q/(f)$. 
In the next part we use these to determine various Frobenius invariants of $R$.

Note that, since $J_q$ is $({\bf x}^q)$-compressed and $q \geq d+3$, the elements $x^q,y^q,z^q$ are part of a minimal generating set for $J_q$. See Remark \ref{rem-qvsd} for details.
Fix a set of minimal generators for $J_q$ as follows 
\[
J_q=(x^q,y^q,z^q,w_1,\dots,w_m)
\qquad {\textrm{so}}
\qquad 
m=\mu(J_q)-3.
\]

The new information employed in this section is one of the main results from Section 4, namely that when $J_q$ is $({\bf x}^q)$-compressed one has $\mu(J_q)=2d+3$ and hence one gets $m=2d$; see Proposition~\ref{numberofgeneratorsofJ}. 

\begin{rmk}
\label{rem-qvsd}
The mild hypothesis $q \geq \deg f +3$ implies that $q \leq \frac{s}{2}$ which ensures that the generators $x^q, y^q, z^q$ lie in degree at most $\frac{s}{2}$; this is implicitly used in several results. In comparison to previous sections, this hypothesis is the $n=3$ case of $q \ge (d+n)/(n-2)$.
\end{rmk}

Recall that we say $f$ is \linkcomp{$q$} if $J_q$ is $({\bf x}^q)$-compressed. 
In view of Remark~\ref{rmk-general}, the result below holds for general choices of $f$ of degree $d<p$, and it holds for all $q$ for very general choices of $f$ of degree $d<p$.

\begin{thm}
\label{gradedBetti}
Let $R=k[x,y,z]/(f)$ be a standard graded hypersurface ring over a field $k$ of characteristic $p>0$ with homogeneous maximal ideal $\fm$. 
Suppose that $p$ and $d=\deg f$ have opposite parity. 
Let $q\geq d+3$ be a power of $p$. 

Assume further that $f$ is \linkcomp{$q$}. 
Then the following hold. 
\begin{itemize}
\item[(a)] 
The matrix $\phi$ in the matrix factorization $(\phi,\phi^\vee)$from Proposition~\ref{shamash} is a $2d\times 2d$ linear matrix with Pfaffian equal to $vf$ for some unit $v\in k$, and its Pfaffian adjoint $\phi^\vee$ is a $2d\times 2d$ matrix with entries of degree $d-1$. 
The matrix $\psi$ is $2d\times 3$ with entries of degree $\frac{1}{2}(q-d+1)$. 
In particular, the minimal graded resolution of $R/\fm^{[q]}$ over $R$ has the following eventually 2-periodic form: 
\[
\cdots \xra{{\phi}} 
R^{2d}(-b-d) \xra{{\phi^\vee}} 
R^{2d}(-b-1) \xra{{\phi}} 
R^{2d}(-b) \xra{\psi^T\phi^\vee} 
R^3(-q) \xra{[{\bf x}^q]} 
R \lra 0 
\]
where $b=\frac{3}{2}q+\frac{1}{2}d-\frac{1}{2}$. 
\vspace{2mm}

\item[(b)] 
The minimal graded resolution of $P/I_q=R/\fm^{[q]}$ over $P$ from Proposition~\ref{I-Presn} has the following form: 
\[
0 \lla P 
\xla{\ \  }  
\begin{matrix}
P^3(-q) \\\oplus\\ P(-d)
\end{matrix}
\xla{\ \ \ } 
\begin{matrix}
P^{2d}(-\frac{3}{2}q-\frac{1}{2}d+\frac{1}{2}) \\\oplus\\ P^3(-q-d)
\end{matrix}
\xla{\ \ \ } 
P^{2d}({\textstyle{-\frac{3}{2}q-\frac{1}{2}d-\frac{1}{2}}}) \lla 0
\]

In particular, the Castelnuovo-Mumford regularity is given by
\begin{equation*}
\reg(R/\fm^{[q]})=\frac{3}{2}q+\frac{1}{2}d-\frac{5}{2}.   
\end{equation*}

\end{itemize} 
\end{thm}

\begin{proof}
We prove (a) in the process of proving (b). 

First note that the hypothesis that $p$ and $d=\deg f$ have opposite parity is equivalent to the socle degree $s=3(q-1)-d$ of $P/J_q$ being even.  

We begin by determining the graded shifts in the resolution of $P/J_q$ from Lemma~\ref{lem-brunsherzog}. 
\[
0 \to P \xrightarrow{\partial_3} P^3\oplus P^m \xrightarrow{\partial_2} P^3\oplus P^m \xrightarrow{\partial_1} P \to 0
\]
where 
\[
\partial_1=[x^q,y^q,z^q,w_1,w_2,\cdots,w_m], 
\qquad 
\partial_2=X=
\begin{bmatrix} 
* & -\psi^T \\ 
\psi & \phi 
\end{bmatrix},
{\textrm{ \ \ and \ \ }}
\partial_3=\partial_1^T. 
\]
As is well-known, the last shift, sometimes called the back twist, is equal to 
\[
-s-3 = -3q+d
\] 
see, for example, \cite[Cor 1.7]{KV-07}.  
But, by the hypotheses,  Proposition~\ref{numberofgeneratorsofJ} yields that $m=2d$ and 
\[
\deg w_i 
=\frac{s}{2}+1 
= \frac{3}{2}q-\frac{1}{2}d-\frac{1}{2}
\]
for each $i=1,\dots,m$.  Therefore the minimal graded resolution has the form 
\[
0 \to P(-3q+d) \xrightarrow{\partial_3} P^3(a)\oplus P^m(b) \xrightarrow{\partial_2} P^3(c)\oplus P^m(e) \xrightarrow{\partial_1} P \to 0
\]
where 
\begin{align*}
a&=-3q+d+(q)=-2q+d,\\ 
b&=-3q+d+(\deg w_i)
= -\frac{3}{2}q+\frac{1}{2}d-\frac{1}{2},\\
c&=-q,   \\
\text{and } e&=-\deg w_i
=-\frac{3}{2}q+\frac{1}{2}d+\frac{1}{2}.
\end{align*}
From this one deduces that $\phi$ is linear (since $e-b=1$). This can also be seen from the fact that $m=2d$ and $\Pf(\phi)=vf$. So its Pfaffian adjoint $\phi^\vee$ is a homogeneous matrix with entries of degree $d-1$. 
In addition, $\psi$ is a homogeneous matrix with \[
\deg \psi 
= e-a
=\frac{1}{2}q-\frac{1}{2}d+\frac{1}{2}.
\]
From this one, in combination with Proposition~\ref{I-Presn}, one gets the desired result. (Equivalently, one can redo the mapping cone argument in the proof there and track the resulting shifts.)
\end{proof}

\vspace{3mm}

\centerline{\bf Part 4: Hilbert-Kunz function} 
\vspace{3mm}
Given a finitely generated graded $k$-algebra $R$ with respect to the homogeneous maximal ideal $\fm$, the Hilbert-Kunz function of $R$ is defined as
\begin{equation*}
    HK_R(q)=\dim_k R/\fm^{[q]}.
\end{equation*}
In \cite{Mo83}, Monsky shows that asymptotically these functions have the form of \begin{equation*}
    HK_R(q)=e_{\textrm{HK}}(R)q^r+\mathcal{O}(q^{r-1}),
\end{equation*}
where $r$ is the Krull dimension of $R$. The constant $e_{\textrm{HK}}(R)$ is called the Hilbert-Kunz multiplicity of $R$.

We compute the Hilbert-Kunz function of $R$ for $f$ \qlinkcomp. In view of Remark~\ref{rmk-general}, the next theorem holds for general choices of $f$ of degree $d<p$, and  holds for all $q$ for very general choices of $f$ of degree $d<p$.

\begin{thm}
\label{HKfunction}
Suppose that $p$ and $d$ are of opposite parity. If $f$ is \qlinkcomp\ for some $q$
then  
\begin{equation*}
HK(q)=\frac{3}{4}dq^2-\frac{1}{12}(d^3-d).   
\end{equation*}
\end{thm}

\begin{proof}
Taking the alternating sum of the Poincar\'e series of the free modules in the resolution of $P/I_q$ from Theorem~\ref{gradedBetti}, one sees that 
\begin{align*}
    \ell(P/I_q)
    &=\!
    \left[
    \frac{1}{(1-t)^3}
    \right]
    -\left[
    \frac{3t^q}{(1-t)^3}
    +\frac{t^d}{(1-t)^3}
    \right]
    +\left[
    \frac{2dt^b}{(1-t)^3}
    +\frac{3t^a}{(1-t)^3}
    \right]
    -
    \left[
    \frac{2dt^{b+1}}{(1-t)^3} 
    \right]
    \\[2mm]
    &=\frac{1-3t^q-t^d+2dt^b+3t^a-2dt^{b+1}}{(1-t)^3}
\end{align*}
where $b=\frac{3}{2}q+\frac{1}{2}d-\frac{1}{2}$ and $a=q+d$. 
To compute the total length of $P/I_q$ (which is the sum of the lengths of the graded pieces) we evaluate this at $t=1$, which we achieve by taking the limit as $t\to 1$ via l'H\^opital's rule. 
This yields 
\begin{align*}
HK(q)=\ell(P/I)=
-\frac{1}{6}
[&-3(q)(q-1)(q-2)
-d(d-1)(d-2) \\
&+2d(b)(b-1)(b-2)
+3(a)(a-1)(a-2) \\
&-2d(b+1)(b)(b-1)]
\\[2mm]
=-\frac{1}{6}
[&-3q(q-1)(q-2)-d(d-1)(d-2) \\
&-6d(b)(b-1)+3a(a-1)(a-2)].
\end{align*}
Plugging in the values of $b$ and $a$ and simplifying yields the desired result. 
\end{proof}

In \cite{BuCh-97}, it is shown that for  general  choices of $f$, the Hilbert-Kunz multiplicity of $R$ achieves the minimum possible value $\frac{3}{4}d$. A similar result follows from  Theorem~\ref{HKfunction} if $(({\bf x}^q) \colon f)$ is $({\bf x}^q)$-compressed for all large enough $q$; 
see Remark~\ref{rmk-general}.

  \section{Examples}
  \label{sec:examples}

In this section we give some examples to illustrate the various types of stability phenomena that occur with respect to Frobenius powers. 
In each we contrast the behavior for a \qlinkcomp\ polynomial with that for a diagonal hypersurface $x^d+y^d+z^d$, as studied in \cite{KRV-12}. 
Furthermore, the non-diagonal polynomial in Example~\ref{diagonalexample} is general enough that it further displays the behavior described in Theorem A in the introduction. 
In fact, it is easy to find such examples since most randomly chosen ones turn out to be generic enough; see Remark~\ref{rmk-general}. 

We set  
\[
P=k[x,y,z]
\ \ {\text{with}} \ \ 
k=\mathbb Z/p\mathbb Z
\ \  {\text{and}} \ \  
R=P/(f).
\]
For each example, we first use  \texttt{Macaulay2}  to compute the degrees of the generators of the linked ideal 
\[
J_q=((x^q,y^q,z^q):(f))
\] 
in order to check whether it is $(x^q,y^q,z^q)$-compressed.  
Indeed, using Lemma~\ref{TFAErelativelycompressed}, we know that the Gorenstein ideal $J_q$ is $(x^{q},y^{q},z^{q})$-compressed as long as its only graded generators in degrees $i \leq \frac{s}{2}$ are $x^{q},y^{q},z^{q}$ (unless $q>\frac{s}{2}$, in which case as long as there are none). 

Then we display the effect on the resolutions by providing some of the graded Betti tables for the $R$-modules $R/\fm^{[q]}$, again computed using \texttt{Macaulay2}. 

\begin{exa}
\label{ex1}
Let $f=xy^2+yz^2+zx^2$. 
First let $p=5$ and so $k=\mathbb{Z}/5\mathbb{Z}$.  Here $p$ and $d$ have the same parity, demonstrating that we can find examples of the cycling behavior beyond the reach of our main results.
Then there are (nongraded) isomorphisms 
\begin{align*}
\syz_3^R(R/(x^5,y^5,z^5))
&\cong \syz_3^R(R/(x^{25},y^{25},z^{25}))
\\
&\cong \syz_3^R(R/(x^{125},y^{125},z^{125}))
\\
&\cong \syz_3^R(R/(x^{625},y^{625},z^{625}))
\\
&\cong \coker \begin{pmatrix} x & z^2 & 0 & 0 \\ -y & y^2+xz & 0 & 0 \\ 2z & 0 & y^2+xz & z^2 \\ 0 & z & -2y & 2x \\ \end{pmatrix}
\end{align*}
and so perhaps the tails of the resolutions are in fact the same for all powers of $p$, up to a graded shift. 

From \texttt{Macaulay2}, we compute the generators of $J_q$ to be as follows. For $e=2$ there are 3 generators of degree 25, namely $x^{25},y^{25},z^{25}$, 3 generators of degree 35, and 1 generator of degree 36. For $e=3$, there are 3 generators of degree 125, namely $x^{125}$, $y^{125}$, $z^{125}$, 3 generators of degree 185, and 1 generator of degree 186. 
So $J_{q}$ is $(x^{q},y^{q},z^{q})$-compressed for at least $q=25$ and $q=125$ and hence exhibit the expected stable Betti number behavior. 

The Betti tables are below. Note that in the Betti tables below the rows that are skipped are zero and the rows that do appear have their row number indicated at the left. 
Note also that since these resolutions are over a hypersurface and hence eventually periodic, we only display columns up to homological degree 4, after which the pattern repeats.

\begin{tabular}{c | c | c | c | c | c}
\multicolumn{6}{c}{$e=1$} \\
\hline
 & 0 & 1 & 2 & 3 & 4 \\
 \hline
 total: & 1 &3  & 4 & 4 & 4 \\
\hline
0: & 1 & . & . & . & . \\
4: & . & 3 & . & . & . \\
6: & . & . & 3 & 1 & . \\
7: & . & . & 1 & 3 & 3 \\
8: & . & . & . & . & 1 \\
\end{tabular}
\hspace{20mm}
\begin{tabular}{c | c | c | c | c | c}
\multicolumn{6}{c}{$e=2$} \\
\hline
 & 0 & 1 & 2 & 3 & 4 \\
 \hline
 total: & 1 &3  & 4 & 4 & 4 \\
\hline
0: & 1 & . & . & . & . \\
24: & . & 3 & . & . & . \\
36: & . & . & 3 & 1 & . \\
37: & . & . & 1 & 3 & 3 \\
38: & . & . & . & . & 1 \\
\end{tabular}

\vspace{1mm}

\vspace{3mm}

However, since $p=5$ and $d=3$ have the same parity, we do not get the behavior described in Theorem A. 

In contrast, take the diagonal hypersurface $f=x^5+y^5+z^5$,  
and let $p=7$, so $k=\mathbb{Z}/7\mathbb{Z}$. 
There are (nongraded) isomorphisms 
\begin{align*}
\syz_3^R(R/(x^7,y^7,z^7))
&\cong \syz_3^R(R/(x^{7^3},y^{7^3},z^{7^3}))
\\
&\cong \coker 
\begin{pmatrix} 
y^3 & x^3 & z^3 & 0 \\
x^2 & -y^2 & 0 & -z^3 \\
z^2 & 0 & -y^2 & x^3 \\
0 & z^2 & -x^2 & -y^3 \\
\end{pmatrix}
\\
\syz_3^R(R/(x^{7^2},y^{7^2},z^{7^2}))
&\cong \syz_3^R(R/(x^{7^4},y^{7^4},z^{7^4}))
\\
&\cong \coker 
\begin{pmatrix} 
z & 0 & x^4 & y^4 \\
x & y^4 & -z^4 & 0 \\
-y & x^4 & 0 & z^4 \\
0 & z & y & -x \\
\end{pmatrix}
\end{align*}
In this case the syzygy cycle appears to have length 2.
The Betti tables are below.

\vspace{1mm}

\begin{tabular}{c | c | c | c | c | c}
\multicolumn{5}{c}{$e=1$} \\
\hline
 & 0 & 1 & 2 & 3 & 4\\
 \hline
 total: & 1 & 3  & 4 & 4 & 4\\
\hline
0: & 1 & . & . & . & .\\
6: & . & 3 & . & . & .\\
9: & . & . & 1 & . & .\\
10: & . & . & 3 & . & .\\
11: & . & . & . & 3 & .\\
12: & . & . & . & 1 & 1\\
13: & . & . & . & . & 3\\
\end{tabular}
\hspace{20mm}
\begin{tabular}{c | c | c | c | c | c}
\multicolumn{5}{c}{$e=2$} \\
\hline
 & 0 & 1 & 2 & 3  & 4\\
 \hline
 total: & 1 &3  & 4 & 4 & 4 \\
\hline
0: & 1 & . & . & . & . \\
48: & . & 3 & . & . & . \\
72: & . & . & 3 & 1 & . \\
75: & . & . & 1 & 3 & 3 \\
78: & . & . & . & . & 1\\
\end{tabular}
\end{exa}

\begin{exa}
\label{diagonalexample}
This example shows how the $(x^q,y^q,z^q)$-compressed case and the diagonal hypersurface cases can differ even for the same values of $p$ and $d$ (here, $p=d=4$). 
First we show that in the diagonal case the link $J_q$ is not $(x^q,y^q,z^q)$-compressed by examining the degrees of its generators and then we provide the resulting Betti tables for the resolutions of $R/\fm^{[q]}$ over $R$. 

Recall that, in the $(x^q,y^q,z^q)$-compressed case, the ideal $J_q$ is generated in degrees $\lceil s/2 \rceil$ and $\lceil s/2 \rceil +1$.
In contrast, the data in the diagonal case is as follows (all computations performed in Macaulay2 \cite{M2}). 

For $p=5$ and $p=7$, in the case of the diagonal hypersurface $x^4+y^4+z^4$, the ideal $J_q$ is not $(x^q,y^q,z^q)$-compressed, which we can see as follows from the degrees of generators of $J_q$. 
For $p=5$, the ideal $((x^5,y^5,z^5):(x^4+y^4+z^4))$ has 1 generator in degree 3 and 6 generators in degree 6. In the $(x^5,y^5,z^5)$-compressed case, it would have generators only in degree 5, which is equal to both $p^e$ and $\lceil \frac{s}{2} \rceil +1$. The ideal $((x^{25},y^{25},z^{25}):(x^4+y^4+z^4))$ has generators in degrees 25, 31, and 36, rather than in degrees $25=p^e$ and $35=\lceil \frac{s}{2} \rceil +1$ as in the $(x^{25},y^{25},z^{25})$-compressed case.
The case of $p=7$ is similar: the ideal $((x^7,y^7,z^7):(x^4+y^4+z^4))$ has six generators of degree 7 and one of degree 9. In the $(x^7,y^7,z^7)$-compressed case, there are generators of degrees 7 ($=p^e$) and 8 ($=\lceil \frac{s}{2} \rceil +1$). The ideal $((x^{49},y^{49},z^{49}):(x^4+y^4+z^4))$ has generators of degrees 49, 70, and 72, rather than $p^e=49$ and $\lceil \frac{s}{2} \rceil+1=71$ as in the $(x^{49},y^{49},z^{49})$-compressed case.

We can see the impact of these differences in the Betti tables for the resolutions of $R/\fm^{[q]}$ over $R$, which we display side-by-side below. The tables below are for the polynomial listed, which appears to be general enough to give a $(x^q,y^q,z^q)$-compressed link $J_q$, and the polynomial $x^4+y^4+z^4$ that behaves differently.
(As in Example~\ref{ex1}, the rows that are skipped are zero, row numbers are indicated at the left, and after column 4 the pattern repeats.) 

First take $p=7$. 
For $e=1$ we have

\vspace{3mm}

\begin{tabular}{c | c | c | c | c | c}
\multicolumn{6}{c}{$f=xy^3+yz^3+zx^3$} \\
\hline
 & 0 & 1 & 2 & 3 & 4 \\
 \hline
 total: & 1 &3  & 8 & 8 & 8 \\
\hline
0: & 1 & . & . & . & . \\
6: & . & 3 & . & . & . \\
10: & . & . & 8 & 8 & . \\
12: & . & . & . & . & 8 \\
\end{tabular}
\hspace{20mm}
\begin{tabular}{c | c | c | c | c | c}
\multicolumn{6}{c}{$f=x^4+y^4+z^4$} \\
\hline
 & 0 & 1 & 2 & 3 & 4 \\
 \hline
 total: & 1 &3  & 4 & 4 & 4 \\
\hline
0: & 1 & . & . & . & . \\
6: & . & 3 & . & . & . \\
9: & . & . & 3 & 1 & . \\
11: & . & . & 1 & 3 & 3 \\
13: & . & . & . & . & 1 \\
\end{tabular}

\vspace{3mm}

and for $e=2$ we have

\vspace{3mm}

\begin{tabular}{c | c | c | c | c | c}
\multicolumn{6}{c}{$f=xy^3+yz^3+zx^3$} \\
\hline
 & 0 & 1 & 2 & 3 & 4 \\
 \hline
 total: & 1 &3  & 8 & 8 & 8 \\
\hline
0: & 1 & . & . & . & . \\
48: & . & 3 & . & . & . \\
73: & . & . & 8 & 8 & . \\
75: & . & . & . & . & 8 \\
\end{tabular}
\hspace{20mm}
\begin{tabular}{c | c | c | c | c | c}
\multicolumn{6}{c}{$f=x^4+y^4+z^4$} \\
\hline
 & 0 & 1 & 2 & 3 & 4 \\
 \hline
 total: & 1 &3  & 4 & 4 & 4 \\
\hline
0: & 1 & . & . & . & . \\
48: & . & 3 & . & . & . \\
72: & . & . & 3 & 1 & . \\
74: & . & . & 1 & 3 & 3 \\
76: & . & . & . & . & 1 \\
\end{tabular}

\vspace{3mm}

Next, we give an example for which the contrast is more extreme for $e=2$. 
Take $p=5$. Indeed, as established in \cite{KRV-12}, for the polynomial $x^4+y^4+z^4$, the module  $R/\fm^{[q]}$ has finite projective dimension over $R$. Compare this to the hypersurface given by the polynomial $f=x^3y-xy^3+x^3z-xz^3-yz^3$, for which the ideal $((x^{25},y^{25},z^{25}):(f))$ has three generators of degree 25, and 8 of degree 35, so it is $(x^{25},y^{25},z^{25})$-compressed. 

For $e=1$ we have

\vspace{3mm}

\begin{tabular}{c | c | c | c | c | c}
\multicolumn{6}{c}{$f=x^3y-xy^3+x^3z-xz^3-yz^3$} \\
\hline
 & 0 & 1 & 2 & 3 & 4 \\
 \hline
 total: & 1 &3  & 4 & 4 & 4 \\
\hline
0: & 1 & . & . & . & . \\
4: & . & 3 & . & . & . \\
7: & . & . & 8 & 8 & . \\
9: & . & . & . & . & 8 \\
\end{tabular}
\hspace{20mm}
\begin{tabular}{c | c | c | c | c | c}
\multicolumn{6}{c}{$f=x^4+y^4+z^4$} \\
\hline
 & 0 & 1 & 2 & 3 & 4 \\
 \hline
 total: & 1 &3  & 4 & 4 & 4 \\
\hline
0: & 1 & . & . & . & . \\
4: & . & 3 & . & . & . \\
5: & . & . & 1 & . & . \\
7: & . & . & 3 & 3 & 1 \\
9: & . & . & . & 1 & 3 \\
\end{tabular}

\vspace{3mm}

and for $e=2$ we have

\vspace{3mm}

\begin{tabular}{c | c | c | c | c | c}
\multicolumn{6}{c}{$f=x^3y-xy^3+x^3z-xz^3-yz^3$} \\
\hline
 & 0 & 1 & 2 & 3 & 4 \\
 \hline
 total: & 1 &3  & 4 & 4 & 4 \\
\hline
0: & 1 & . & . & . & . \\
24: & . & 3 & . & . & . \\
37: & . & . & 8 & 8 & . \\
39: & . & . & . & . & 8 \\
\end{tabular}
\hspace{20mm}
\begin{tabular}{c | c | c | c| c | c}
\multicolumn{6}{c}{$f=x^4+y^4+z^4$} \\
\hline
 & 0 & 1 & 2 & 3 & 4\\
 \hline
 total: & 1 &3  & 2 &0 & 0\\
\hline
0: & 1 & . & . & . & .\\
24: & . & 3 & . & . & .\\
33: & . & . & 1 & . & .\\
38: & . & . & 1  &. & .\\
\multicolumn{5}{c}{ }
\end{tabular}

\vspace{3mm}

\noindent 
For $e=2$, for the diagonal hypersurface,  $R/\fm^{[q]}$ even has finite projective dimension over $R$.
In contrast, the other polynomial is general enough that the resolution exhibits the behavior described in Theorem A of the introduction, that is, one of the matrices in the $8 \times 8$ matrix factorization is actually linear (and is symmetric with Pfaffian equal to a unit times $f$); furthermore, the high Betti numbers are the maximal possible, namely $2d$. 
\end{exa}

\section*{Acknowledgements}
We would like to thank Luchezar Avramov for bringing \cite{Av81} to our attention, which was crucial in leading us to prove our results in Section 5.

\bibliographystyle{amsalpha}
\bibliography{citations}

\def\cprime{$'$} \def\cprime{$'$}
\providecommand{\bysame}{\leavevmode\hbox to3em{\hrulefill}\thinspace}
\providecommand{\MR}{\relax\ifhmode\unskip\space\fi MR }
\providecommand{\MRhref}[2]{%
  \href{http://www.ams.org/mathscinet-getitem?mr=#1}{#2}
}
\providecommand{\href}[2]{#2}
\begin{thebibliography}{MMRN11}

\bibitem[Avr78]{Av78}
Luchezar~L. Avramov, \emph{Small homomorphisms of local rings}, J. Algebra
  \textbf{50} (1978), no.~2, 400--453. \MR{485906}

\bibitem[Avr81]{Av81}
\bysame, \emph{Poincar\'{e} series of almost complete intersections of
  embedding dimension three}, PLISKA Stud. Math. Bulgar. \textbf{2} (1981),
  167--172. \MR{633871}

\bibitem[Avr10]{Av10}
\bysame, \emph{Infinite free resolutions}, Six lectures on commutative algebra,
  Mod. Birkh\"{a}user Class., Birkh\"{a}user Verlag, Basel, 2010, pp.~1--118.
  \MR{2641236}

\bibitem[BC97]{BuCh-97}
Ragnar-Olaf Buchweitz and Qun Chen, \emph{Hilbert-{K}unz functions of cubic
  curves and surfaces}, J. Algebra \textbf{197} (1997), no.~1, 246--267.
  \MR{1480784}

\bibitem[BE75]{BuchEis-75}
David~A. Buchsbaum and David Eisenbud, \emph{Generic free resolutions and a
  family of generically perfect ideals}, Advances in Math. \textbf{18} (1975),
  no.~3, 245--301. \MR{0396528}

\bibitem[BE77]{BE-77}
\bysame, \emph{Algebra structures for finite free resolutions, and some
  structure theorems for ideals of codimension {$3$}}, Amer. J. Math.
  \textbf{99} (1977), no.~3, 447--485. \MR{0453723}

\bibitem[BH93]{bruns-herzog}
Winfried Bruns and J{\"u}rgen Herzog, \emph{Cohen-{M}acaulay rings}, Cambridge
  Studies in Advanced Mathematics, vol.~39, Cambridge University Press,
  Cambridge, 1993. \MR{1251956 (95h:13020)}

\bibitem[BL94]{BoLak-94}
Mats Boij and Dan Laksov, \emph{Nonunimodality of graded {G}orenstein {A}rtin
  algebras}, Proc. Amer. Math. Soc. \textbf{120} (1994), no.~4, 1083--1092.
  \MR{1227512}

\bibitem[Cam23]{Camphire-thesis}
Heath Camphire, \emph{Determining the {B}etti numbers of ${R}/(x^{p^e},
  y^{p^e}, z^{p^e})$ for most even degree hypersurfaces in odd characteristic},
  ProQuest LLC, 2023, Thesis (Ph.D.)--George Mason University.

\bibitem[Eis95]{Eise-95}
David Eisenbud, \emph{Commutative algebra with a view toward algebraic
  geometry}, Graduate Texts in Mathematics, vol. 150, Springer-Verlag, New
  York, 1995. \MR{97a:13001}

\bibitem[EKK14]{ElkhKus-14}
Sabine El~Khoury and Andrew~R. Kustin, \emph{Artinian {G}orenstein algebras
  with linear resolutions}, J. Algebra \textbf{420} (2014), 402--474.
  \MR{3261467}

\bibitem[Ems78]{Ems-78}
Jacques Emsalem, \emph{G\'eom\'etrie des points \'epais}, Bull. Soc. Math.
  France \textbf{106} (1978), no.~4, 399--416. \MR{518046}

\bibitem[GS]{M2}
Daniel~R. Grayson and Michael~E. Stillman, \emph{Macaulay2, a software system
  for research in algebraic geometry}, Available at
  \url{http://www.math.uiuc.edu/Macaulay2/}.

\bibitem[Iar84]{Iar-84}
Anthony Iarrobino, \emph{Compressed algebras: {A}rtin algebras having given
  socle degrees and maximal length}, Trans. Amer. Math. Soc. \textbf{285}
  (1984), no.~1, 337--378. \MR{748843}

\bibitem[IK99]{IarKan-99}
Anthony Iarrobino and Vassil Kanev, \emph{Power sums, {G}orenstein algebras,
  and determinantal loci}, Lecture Notes in Mathematics, vol. 1721,
  Springer-Verlag, Berlin, 1999, Appendix C by Iarrobino and Steven L. Kleiman.
  \MR{1735271}

\bibitem[Kat98]{Kat-98}
Mordechai Katzman, \emph{The complexity of {F}robenius powers of ideals}, J.
  Algebra \textbf{203} (1998), no.~1, 211--225. \MR{1620654}

\bibitem[KRGV22]{KRGV-22}
Andrew~R. Kustin, Rebecca R.~G., and Adela Vraciu, \emph{The resolution of
  {$(x^N, y^N, z^N, w^N)$}}, J. Algebra \textbf{590} (2022), 338--393.
  \MR{4332033}

\bibitem[KRGV23]{KRGV-23}
\bysame, \emph{The syzygies of the ideal {$(x^N_1, x^N_2, x^N_3, x^N_4)$} in
  the hypersurface ring defined by {$x^n_1+x^n_ 2+x^n_3+x^n_ 4$}}, J. Algebra
  \textbf{615} (2023), 205--242. \MR{4506076}

\bibitem[KRV12]{KRV-12}
Andrew~R. Kustin, Hamidreza Rahmati, and Adela Vraciu, \emph{The resolution of
  the bracket powers of the maximal ideal in a diagonal hypersurface ring}, J.
  Algebra \textbf{369} (2012), 256--321. \MR{2959795}

\bibitem[KU09]{KU-09}
Andrew~R. Kustin and Bernd Ulrich, \emph{Socle degrees, resolutions, and
  {F}robenius powers}, J. Algebra \textbf{322} (2009), no.~1, 25--41.
  \MR{2526373}

\bibitem[KV07]{KV-07}
Andrew~R. Kustin and Adela~N. Vraciu, \emph{Socle degrees of {F}robenius
  powers}, Illinois J. Math. \textbf{51} (2007), no.~1, 185--208. \MR{2346194}

\bibitem[MMRN05]{MMN-05}
Juan~C. Migliore, Rosa~M. Mir\'o-Roig, and Uwe Nagel, \emph{Minimal resolution
  of relatively compressed level algebras}, J. Algebra \textbf{284} (2005),
  no.~1, 333--370. \MR{2115019}

\bibitem[MMRN11]{MMN-11}
\bysame, \emph{Monomial ideals, almost complete intersections and the weak
  {L}efschetz property}, Trans. Amer. Math. Soc. \textbf{363} (2011), no.~1,
  229--257. \MR{2719680}

\bibitem[Mon83]{Mo83}
P.~Monsky, \emph{The {H}ilbert-{K}unz function}, Math. Ann. \textbf{263}
  (1983), no.~1, 43--49. \MR{697329}

\bibitem[MR18]{MiRa-18}
Claudia Miller and Hamidreza Rahmati, \emph{Free resolutions of {A}rtinian
  compressed algebras}, J. Algebra \textbf{497} (2018), 270--301. \MR{3743182}

\bibitem[Nor74]{Nor-74}
D.~G. Northcott, \emph{Injective envelopes and inverse polynomials}, J. London
  Math. Soc. (2) \textbf{8} (1974), 290--296. \MR{0360555}

\bibitem[PS74]{PesSzp-74}
C.~Peskine and L.~Szpiro, \emph{Liaison des vari\'{e}t\'{e}s alg\'{e}briques.
  {I}}, Invent. Math. \textbf{26} (1974), 271--302. \MR{0364271}

\bibitem[Sha69]{JSh69}
Jack Shamash, \emph{The {P}oincar\'e series of a local ring}, J. Algebra
  \textbf{12} (1969), 453--470. \MR{MR0241411 (39 \#2751)}

\end{thebibliography}


\end{document}